\DeclareSymbolFont{AMSb}{U}{msb}{m}{n}
\documentclass[11pt,a4paper,twoside,leqno,noamsfonts]{amsart}
           \usepackage{setspace}
\usepackage[english]{babel}
\usepackage[dvipsnames]{xcolor}
\definecolor{britishracinggreen}{rgb}{0.0, 0.26, 0.15}
\definecolor{cobalt}{rgb}{0.0, 0.28, 0.67}
\usepackage[utopia]{mathdesign} 
  \DeclareSymbolFont{usualmathcal}{OMS}{cmsy}{m}{n}
  \DeclareSymbolFontAlphabet{\mathcal}{usualmathcal}
\usepackage[a4paper,top=3cm,bottom=3cm,left=3cm,right=3cm,bindingoffset=5mm]{geometry}
\usepackage[utf8]{inputenc}
\usepackage{braket,caption,comment}
\usepackage{multirow,booktabs,microtype}
\usepackage{todonotes}
\usepackage{breqn}
  \numberwithin{equation}{section}
\usepackage[colorlinks,bookmarks]{hyperref} %bookmarks,
\hypersetup{colorlinks,%
citecolor=britishracinggreen,%
filecolor=black,%
linkcolor=cobalt,%
urlcolor=black}


\def\be{\begin{equation}}    
\def\ee{\end{equation}}
\def\bitem{\begin{itemize}}
\def\eitem{\end{itemize}}
\def\benum{\begin{enumerate}}
\def\eenum{\end{enumerate}}

\def\ra{\rightarrow}
\def\lra{\longrightarrow}

\def\mbgp{\overline{\mathcal M}_{g,P}}
\def\jbgp{\overline{\mathcal J}_{g,P}}
\def\cbgp{\overline{\mathcal C}_{g,P}}

\def\Z{\mathbb Z}

\def\M{\overline{\mathcal M}}

\def\O{\mathscr O}

\DeclareMathOperator{\Fit}{Fit}

\DeclareMathOperator{\ch}{ch}
\DeclareMathOperator{\Td}{Td}

\DeclareMathOperator{\Pic}{Pic}

\DeclareMathOperator{\Aut}{Aut}

\DeclareMathOperator{\rk}{rk}

\theoremstyle{definition}

\newtheorem*{lemma*}{Lemma}
\newtheorem*{theorem*}{Theorem}
\newtheorem*{example*}{Example}
\newtheorem*{fact*}{Fact}
\newtheorem*{notation*}{Notation}
\newtheorem*{definition*}{Definition}
\newtheorem*{prop*}{Proposition}
\newtheorem*{remark*}{Remark}
\newtheorem*{corollary*}{Corollary}
\newtheorem*{conventions*}{Conventions}

\newtheorem{definition}{Definition}[section]
\newtheorem{notation}[definition]{Notation}

\newtheorem{convention}[definition]{Convention}
\newtheorem{example}[definition]{Example}

\newtheorem{remark}[definition]{Remark}

\newtheoremstyle{thm} % <name> % (ambienti con dimostrazione)
        {3mm}% <Space above>
        {3mm}% <Space below>
        {\slshape}% <Body font> % 
        {0mm}% <Indent amount>
        {\bfseries}% <Theorem head font>
        {.}% <Punctuation after theorem head>
        {1mm}% <Space after theorem head>
        {}% <Theorem head spec (can be left empty, meaning 'normal')> 
\theoremstyle{thm}

\newtheorem{corollary}[definition]{Corollary}
\newtheorem{lemma}[definition]{Lemma}

\newtheorem{thm}{Theorem}
\usepackage{tikz}
\usepackage{tikz-cd}
\usepackage{rotating}

\usetikzlibrary{matrix,shapes,arrows,decorations.pathmorphing}
\tikzset{commutative diagrams/arrow style=math font}
\tikzset{commutative diagrams/.cd,
mysymbol/.style={start anchor=center,end anchor=center,draw=none}}
\newcommand\MySymb[2][\square]{%
  \arrow[mysymbol]{#2}[description]{#1}}
\tikzset{
shift up/.style={
to path={([yshift=#1]\tikztostart.east) -- ([yshift=#1]\tikztotarget.west) \tikztonodes}
}
}

\DeclareMathAlphabet{\mathpzc}{OT1}{pzc}{m}{it}

\newcommand*{\defeq}{\mathrel{\vcenter{\baselineskip0.5ex \lineskiplimit0pt
                     \hbox{\scriptsize.}\hbox{\scriptsize.}}}%
                     =}

\usepackage{tkz-graph}
\usetikzlibrary{matrix}
 \usetikzlibrary{shapes}
\usetikzlibrary{arrows,decorations.markings}

\pgfarrowsdeclare{bad to}{bad to}
{
  \pgfarrowsleftextend{-2\pgflinewidth}
  \pgfarrowsrightextend{\pgflinewidth}
}
{
  \pgfsetlinewidth{0.8\pgflinewidth}
  \pgfsetdash{}{0pt}
  \pgfsetroundcap
  \pgfsetroundjoin
  \pgfpathmoveto{\pgfpoint{-3\pgflinewidth}{4\pgflinewidth}}
  \pgfpathcurveto
  {\pgfpoint{-2.75\pgflinewidth}{2.5\pgflinewidth}}
  {\pgfpoint{0pt}{0.25\pgflinewidth}}
  {\pgfpoint{0.75\pgflinewidth}{0pt}}
  \pgfpathcurveto
  {\pgfpoint{0pt}{-0.25\pgflinewidth}}
  {\pgfpoint{-2.75\pgflinewidth}{-2.5\pgflinewidth}}
  {\pgfpoint{-3\pgflinewidth}{-4\pgflinewidth}}
  \pgfusepathqstroke
}

\title[Pullbacks of universal Brill--Noether classes via Abel--Jacobi morphisms]{Pullbacks of universal Brill--Noether classes \\ via Abel--Jacobi morphisms}
\author[N. Pagani]{Nicola Pagani}
\address{University of Liverpool}
\email[Nicola Pagani]{pagani@liv.ac.uk}

\author[A. T. Ricolfi]{Andrea T. Ricolfi}
\address{SISSA Trieste}
\email[Andrea Ricolfi]{aricolfi@sissa.it}

\author[J. van Zelm]{Jason van Zelm}
\address{Humboldt Universit\"{a}t Berlin}
\email[Jason van Zelm]{jasonvanzelm@outlook.com}

\begin{document}
\maketitle

\begin{abstract} Following Mumford and Chiodo, we compute the Chern character of the derived pushforward $\ch (R^\bullet\pi_\ast\O(\mathsf D))$, for $\mathsf D$ an arbitrary element of the Picard group of the universal curve over the moduli stack of stable marked curves. This allows us to express the pullback of universal Brill--Noether classes via Abel--Jacobi sections to the compactified universal Jacobians, for all compactifications such that the section is a well-defined morphism.
\end{abstract}

{\hypersetup{linkcolor=black}
\tableofcontents}

\section{Introduction}
Let $\pi \colon \, \cbgp \to \mbgp$ be the universal curve over the moduli stack of stable marked curves, where $P$ is a nonempty set of markings. The (weak version of) Franchetta's conjecture, now a theorem due to Harer \cite{Harer} and Arbarello--Cornalba \cite{AC87}, gives an explicit description of the Picard group of the universal curve. Every divisor on $\cbgp$, up to a divisor pulled back from $\mbgp$, is linearly equivalent to
\begin{equation} \label{univdivisor}
\mathsf D=\ell\widetilde{K}_{\pi}+ \sum_{p\in P} d_p \sigma_p + \sum_{h,S} a_{h,S} C_{h,S}
\end{equation}
for some integers $\ell$, $d_p$ and $a_{h,S}$. Here $\widetilde{K}_{\pi}= c_1 (\omega_{\pi})$ is the first Chern class of the relative dualising sheaf, $\sigma_p$ is the class of the $p$-th section, and $C_{h,S}$ (see Definition \ref{definition_CHS}) is the class of the irreducible component not containing the moving point lying above the boundary divisor $\Delta_{h,S} \subset \mbgp$ (more details in Section~\ref{Notation}).

Our main result is an explicit formula for the Chern character of the derived pushforward
\[
\ch (R^\bullet\pi_\ast\O(\mathsf D)),
\]
in terms of certain standard generators of the tautological ring (boundary strata classes decorated with $\kappa$ classes and $\psi$ classes). These generators, denoted $\mathbf{X},  \widetilde{\mathbf{X}}, \widetilde{\mathbf{Y}}$ and $\mathbf{Z}$, are introduced in Notation~\ref{notation_tautological}. To state our main result we first recall the definition of the Bernoulli polynomials $B_t(\ell)$, which are defined by the identity
\[
  \sum_{t\geq 0}\frac{B_t(\ell)}{t!}x^t\defeq e^{\ell x}\frac{x}{e^x-1}.
\]
In particular, $B_t \defeq B_t(0)$ are the classical Bernoulli numbers.

\smallbreak
In Section~\ref{formulazza} we prove:

\begin{thm}\label{MainThm}
If $\mathsf D$ is as in \eqref{univdivisor}, then
\[
\ch(R^\bullet\pi_\ast\O(\mathsf D)) = \Omega + \Phi,
\]
where
\begin{multline*}
    \Omega = \sum_{\substack{t\geq 1\\a+b=t}} \frac{B_b(\ell)}{b!}\sum_{\substack{r\geq 0\\k_1+\cdots +k_r = a \\ k_j>0 \\ (h_1,S_1)< \cdots < (h_r,S_r)}} \left(\prod_{j=1}^{r} \frac{a_{h_j,S_j}^{k_j}}{k_j!} \right) \mathbf{Z}_{(\mathbf{h}_r),(\mathbf{S}_r)}^{\mathbf{k}_r,b-1}  \\ 
 + \sum_{\substack{t\geq 1\\ a+b=t \\ b>0 \\ \alpha+\beta = b}}\frac{(-1)^\beta B_\beta(\ell)}{\alpha!\beta!}\sum_{\substack{r\geq 0\\k_1+\cdots +k_r = a \\ k_j>0 \\ (h_1,S_1)< \cdots < (h_r,S_r)}}\left(\prod_{j=1}^{r} \frac{a_{h_j,S_j}^{k_j}}{k_j!} \sum_{p\in P\setminus S_r} d_p^\alpha(-\psi_p)^{b-1}\right)\mathbf{X}_{(\mathbf{h}_r),(\mathbf{S}_r)}^{\mathbf{k}_r} 
\end{multline*}
and 
 \begin{multline*}
 \Phi=
 \sum_{t\geq 2} \sum_{\substack{a+b=t\\ b>0\textup{ even}\\ a\geq 0}} \frac{ B_{b}}{b!} \sum_{\substack{r \geq 0\\k_1+\cdots +k_r = a \\ k_j>0 \\ (h_1,S_1)< \cdots < (h_r,S_r)}}
   \prod_{j=1}^r  \frac{a_{h_j,S_j}^{k_j}}{k_j!} \sum_{ 0 \leq e \leq b-2} (-1)^{e}\cdot \\
   \cdot \Bigg( \left(\sum_{(l,T)> (h_r,S_r)}\widetilde{\mathbf{X}}_{(\mathbf{h}_r,l),(\mathbf{S}_r,T)}^{\mathbf{k}_r,(e,b-2-e)}\right)  + \widetilde{\mathbf{Y}}_{(\mathbf{h}_r),(\mathbf{S}_r)}^{\mathbf{k}_r,(e,b-2-e)} 
    +  (-1)^{k_r} \widetilde{\mathbf{X}}_{(\mathbf{h}_{r-1},h_r),(\mathbf{S}_{r-1},S_r)}^{\mathbf{k}_{r-1},(e+k_r,b-2-e)}  \Bigg)
\end{multline*}
and the symbol $(h_1,S_1)<\cdots<(h_r,S_r)$ denotes a strictly ordered chain of stable bipartitions (see Notation~\ref{firstnotation}). 
\end{thm}

Our formula expresses $\ch_t(R^\bullet\pi_\ast\O(\mathsf D))$ as a polynomial of degree $t+1$ in the variables $\ell$, $d_p$, $a_{h,S}$ with coefficients in the tautological ring of $\mbgp$. The special case where all $a_{h,S} = 0$ can be extracted from Chiodo's formula \cite[Theorem~1.1.1]{Chiodo}.
 
 We prove Theorem \ref{MainThm} by applying the Grothendieck--Riemann--Roch formula to the universal curve $\pi$, as in Mumford's seminal calculation of the Chern character of the Hodge bundle \cite[Section~4]{Mumford1983}.
 
The formula in Theorem~\ref{MainThm} has been implemented into the Sage program \cite{admcycles} and is available upon request from the third named author.

\medskip
Our main motivation is computing the pullback of (extended, cohomological) Brill--Noether classes $\mathsf w^r_d$ on the universal Jacobian via the Abel--Jacobi sections. Here we give a preview, full details are in Section \ref{Jacobians}. 

Fix $0 \leq d \leq g-1$ and let $\mathcal{J}_{g,P}^d\ra \mathcal M_{g,P}$ be the universal Jacobian parametrising line bundles of degree $d$ over smooth $P$-pointed curves of genus $g$. Let $\mathscr L$ denote the universal (or Poincar\'e) line bundle on the universal curve 
\[
\widetilde{\pi} \colon \, \mathcal{J}_{g,P}^d \times_{\mathcal{M}_{g,P}} \mathcal{C}_{g,P} \to \mathcal{J}_{g,P}^d.
\]
For $0 \leq r \leq d/2$, the universal Brill--Noether locus $\mathcal W^r_d$ is set-theoretically defined by
\[
\mathcal W^r_d \defeq \Set{(C, P, L) | L\in \Pic^dC,\, h^0(C,L)>r} \subset \mathcal{J}_{g,P}^d,
\]
and can be endowed with the scheme structure of the $(g-d+r$)-th Fitting ideal of $R^1 \widetilde{\pi} _\ast\mathscr L$. Each $\mathcal W^r_d$  is in general not equidimensional, and the dimension of its irreducible components is unknown. However, a cohomological
Brill--Noether class $\mathsf w_d^r$ supported on $\mathcal W^r_d$ and of the expected dimension can be defined, via the Thom--Porteous formula, as the $(r+1)\times (r+1)$ determinant
\begin{equation} \label{wrd}
\mathsf w_d^r = 
\Delta_{g-d+r}^{(r+1)} c (-R^\bullet \widetilde{\pi}_\ast\mathscr L)
\in A^\bullet(\mathcal{J}_{g,P}^d).
\end{equation}
The notation $\Delta^{(p)}_{q}c$ stands for the $p\times p$ determinant $|c_{q+j-i}|$, for $1\leq i,j\leq p$ and a general series $c = \sum_kc_k$ (see Section~\ref{PBBN} for more details).

The discussion of the previous paragraph extends verbatim to $\mbgp$. One constructs a compactified universal Jacobian 
\be\label{JB}
\jbgp(\phi) \to \mbgp
\ee
for all nondegenerate polarisations $\phi$, and classes $\mathsf w^r_d(\phi)$ also defined by Formula~\eqref{wrd}, \emph{mutatis mutandis}. The compactified universal Jacobian \eqref{JB} extends $\mathcal{J}_{g,P}^d\ra \mathcal M_{g,P}$ and consists of torsion free sheaves of rank $1$ on stable curves, whose multidegree is prescribed by $\phi$. The rational sections of \eqref{JB} are called \emph{Abel--Jacobi sections}. By Franchetta's conjecture, they are all of the form
\begin{equation} \label{aj}
\mathsf s \colon \,  (C, P) \mapsto \omega_{\pi}^\ell \left( \sum_{p\in P} d_p \sigma_p + \sum_{h,S} a_{h,S} C_{h,S} \right),
\end{equation}
for some integers $\ell$, $d_p$ and $a_{h,S}$. 

A natural question that has attracted lots of attention is computing the pullback of $\mathsf w^r_d(\phi)$ via the section $\mathsf s$. This problem is complicated by the fact that the latter section  is, in general, only a rational map. Theorem~\ref{MainThm} allows one to compute $\mathsf s^\ast\mathsf w^r_d(\phi)$ for every $\phi$ such that $\mathsf s$ is a morphism (these $\phi$'s are characterised in \cite[Section~6.1]{KassPagani1}). Indeed, for every such $\phi$, we will prove in Corollary~\ref{cor:PB} the equality
\begin{equation} \label{PBsigma}
\mathsf s^*\mathsf w^r_d(\phi) = \Delta_{g-d+r}^{(r+1)}c(-R^\bullet\pi_\ast\O(\mathsf D(\phi))),
\end{equation}
where $\mathsf D(\phi)$ is a modification of a divisor $\mathsf D$ as in \eqref{univdivisor} obtained by replacing the coefficients $a_{h,S}$ with the unique coefficients $a_{h,S}(\phi)$ such that $\mathsf D(\phi)$ is $\phi$-stable on all curves with $1$~node.  Combining~\eqref{PBsigma} with Theorem~\ref{MainThm} and with the inversion formula (see Equation~\eqref{inversion}) for the Chern character, we obtain an explicit expression, for all $\phi$ such that $\mathsf s$ is a morphism, for the cohomology class $\mathsf s^*\mathsf w^r_d(\phi)$ in terms of the standard generators of the tautological ring.

The case $r=d=0$ is related to the problem of extending and calculating the ($\ell$-twisted) \emph{Double Ramification Cycle} --- more details are in Section~\ref{drc2} (see also Example~\ref{drc1}).

\begin{conventions*}
We will work over the field of complex numbers $\mathbb{C}$. If $X$ is a smooth Deligne--Mumford stack, we will denote by $A^{\bullet}(X)$ its Chow ring with rational coefficients. 
\end{conventions*}

%%%%%%%%%%%%%%%%%%%%%%%%%
\section{Tautological classes}
\label{Notation}

\subsection{Definition of the tautological ring}
Throughout we fix an integer $g \geq 1 $ and a set of markings $P \neq \emptyset$.
We follow the exposition and the notation of \cite[Section~17.4]{ACG} to introduce the \emph{tautological ring} of the moduli space $\overline{\mathcal M}_{g,P}$ of stable $P$-pointed curves of genus $g$.

It is well-known that the universal curve over the moduli stack of stable $P$-pointed curves can be identified with the forgetful morphism from the moduli stack with one extra marking. Throughout we will denote them by
\[
\cbgp= \overline{\mathcal{M}}_{g,P \cup \set{x}} \xrightarrow{\pi} \mbgp,
\]
and we will freely switch from one description to the other.  

For each marking $p\in P$, we let 
\[
\sigma_p\in A^1(\cbgp)
\]
denote the divisor class corresponding to the $p$-th section of $\pi$. Let $\omega_\pi$ be the relative dualising sheaf, and set
\[
K_\pi \defeq c_1\left(\omega_\pi\left(\sum_p\sigma_p\right)\right), \qquad \widetilde K_\pi \defeq c_1(\omega_\pi) = K_\pi-\sum_p\sigma_p.
\]
We define the cotangent line classes by 
\[
\psi_p \defeq \sigma_p^\ast \widetilde{K}_{\pi}\in A^1(\mbgp). 
\]
For $a\geq 0$, we define the \emph{kappa classes}
\[
\kappa_a \defeq \pi_\ast K_\pi^{a+1} \,\in\, A^a(\mbgp).
\]
The \emph{tautological ring} of the moduli space of stable marked curves
\[
R^\bullet (\overline{\mathcal M}_{g,P})\subset A^\bullet (\overline{\mathcal M}_{g,P})
\]
was originally defined by Mumford in \cite[Section 4]{Mumford1983} in the unmarked case $P=\emptyset$ (which is not discussed in this paper), and an elegant definition for all moduli spaces of stable marked curves at once was later given by C.~Faber and R.~Pandharipande \cite{FPtaut}. We will give here an alternative definition to suit our purposes. 

First we recall the notion of decorated boundary stratum class. For $\Gamma=(\mathrm V(\Gamma),\mathrm E(\Gamma),\mathrm L(\Gamma))$ in the set $\mathsf G_{g,P}$ of isomorphism classes of stable $P$-pointed graphs of genus $g$ (see \cite[Chapter XII.10]{ACG} for the precise definition of a stable graph and of the set $\mathsf G_{g,P}$), we let
\[
\overline{\mathcal M}_\Gamma = \prod_{v\in \mathrm V(\Gamma)}\overline{\mathcal M}_{\mathrm g_v,\mathrm P_v}
\]
and denote by $\xi_\Gamma \colon \, \overline{\mathcal M}_\Gamma\ra \overline{\mathcal M}_{g,P}$ the associated clutching morphism. Here, $\mathrm P_v$ is the set of half-edges and legs issuing from the vertex $v$, and we require that the stability condition $2\mathrm g_v-2+|\mathrm P_v|\,\,>0$ is fulfilled for all vertices $v$. A ``decoration'' $\theta = (\theta_v)_v$ on the graph $\Gamma$ is the datum of a monomial
\[
\theta_v=\prod_{p \in \mathrm P_v} \psi_p^{a_p} \prod_{j} \kappa_j^{b_j}\in A^\bullet(\overline{\mathcal M}_{\mathrm g_v,\mathrm P_v})
\]
for each vertex $v\in \mathrm V(\Gamma)$.
Classes of the form 
\[
\frac{1}{|\Aut \Gamma|} \xi_{\Gamma\ast}\left(\prod_{v \in \mathrm{V}(\Gamma)} \theta_v \right) \in A^\bullet(\mbgp),
\] 
for $\Gamma$ and $\theta$ as above, are called \emph{decorated boundary strata classes}. (Here and in the following, we omit writing the pullback via the projection map to the factor, and we omit writing the tensor product of classes, unless that helps identifying which factor they are pulled back from). We define $R^\bullet (\overline{\mathcal M}_{g,P})$ to be the vector subspace of  $A^\bullet(\overline{\mathcal M}_{g,P})$ generated by these classes and then endow it with the intersection product. When $\theta_v$ is trivial for all $v$, we simply write $\delta_\Gamma \defeq \xi_{\Gamma\ast}(\mathbb 1)/|\Aut \Gamma|$.

The collection of decorated boundary strata classes can be made into a finite set (for fixed $g$ and $P$) by only considering decorations $\theta$ that are not obviously vanishing for degree reasons. Even so, this collection is far from being a basis. All known relations among these generators belong to a vector space generated by the so-called \emph{Pixton's relations}, see \cite{PPZ} and \cite{Janda}, but whether or not these are all the existing relations is so far unknown.

In this paper, ``calculating'' an element of $R^\bullet (\overline{\mathcal M}_{g,P})$  always means expressing it as an explicit, non-unique, linear combination of decorated boundary strata classes.   
We will often use graph notation for these  classes; 
for example we will denote by 
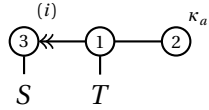
\begin{center}
    \begin{tikzpicture}[->,>=bad to,baseline=-3pt,node distance=1.3cm,thick,main node/.style={circle,draw,font=\Large,scale=0.5}]
\node at (0,0) [main node] (A) {$3$};
\node at (1,0) [main node] (B) {$1 $};
\node at (2,0) [main node] (C) {$2$};
\node at (0,-.7)  (As) {$S$};
\node at (1,-.7)  (Bs) {$T $};
\node at (.3,.4) (p) {\tiny{$(i)$}};
\node at (2.3,.3) (p) {\tiny{$\kappa_a$}};
\draw [<<-] (A) to (B);
\draw [-] (B) to (C);
\draw [-] (A) to (As);
\draw [-] (B) to (Bs);
\end{tikzpicture}
\end{center}
the class  $\xi_{\Gamma * } (\psi_{p_1}^i\otimes \mathbb{1} \otimes \kappa_a)$, where $\xi_\Gamma$ is the clutching morphism \[\M_{3,S\cup \{p_1\}} \times \M_{1,T\cup \{p_2,p_3\}} \times \M_{2,\{p_4\}}\longrightarrow \M_{6,S\cup T}\] which glues $p_1$ to $p_2$ and $p_3$ to $p_4$.

%%%%%%%%%%%%%%%%%%%%%%%%%%%%%%%%%%
\subsection{Boundary divisors}
Here we discuss and fix some convention for the particular case of the tautological classes that correspond to boundary divisors.

\begin{definition}\label{Ex:stablebip}
We define the set of \emph{stable bipartitions} of $(g,P)$ to be the collection of pairs $(h,S)$ where $S\subseteq P$ is a subset of the set of markings, and $0\leq h\leq g$ is such that if $h=0$ then $|S|\,\, \geq 2$ and if $h=g$ then $|S^c|\,\, \geq 2$ (where $S^c=P\setminus S$ denotes the complement). 
\end{definition}

We also make the following:

\begin{convention}\label{convention_1isinS}
We assume that for every stable bipartition $(h,S)$, the set $S$ contains a distinguished marking $1\in P$. (In particular, $S$ is never empty.)
\end{convention}
With this convention, there is a bijection between the set of stable bipartitions and the set of stable graphs $\Gamma_{h,S}\in \mathsf G_{g,P}$ with two vertices and one edge.

The (codimension one) clutching morphism corresponding to $\Gamma_{h,S}$ is denoted
\[
\xi_{h,S}\colon \, \overline{\mathcal M}_{h,S\cup \set{q}}\times \overline{\mathcal M}_{g-h,S^c\cup \set{r}}\ra \mbgp.
\]
Its image is the boundary divisor $\Delta_{h,S}$ and its class $\delta_{\Gamma_{h,S}}$ will simply be denoted by $\delta_{h,S}$.

There is one more boundary divisor of $\mbgp$, which parametrises irreducible singular curves. That divisor is the image of the clutching morphism $\xi_{\text{irr}}$ that corresponds to the stable graph $\Gamma_{\text{irr}}$ consisting of one vertex of genus $g-1$ with a loop and with all markings $P$. 

\begin{definition}\label{definition_CHS}
For a fixed stable bipartition $(h,S)$ of $(g,P)$, the inverse image $\pi^{-1}(\Delta_{h,S})$ in the universal curve $\cbgp$ consists of two irreducible components. We will denote by $C^+_{h,S}$ the class of the component that contains the moving point $x$ on the universal curve, and by $C_{h,S}$ the class of the other component, see Figure \ref{fig:chs}.
\end{definition}

\begin{figure}[!h]
    \centering
\begin{tikzpicture}[->,>=bad to,baseline=-3pt,node distance=1.3cm,thick,main node/.style={circle,draw,font=\Large,scale=0.5}]
\node at (0,0) [main node] (A) {$h$};
\node at (2,0) [main node] (B) {$g-h$};
\node at (0,-1)  (As) {$S$};
\node at (2,-1)  (Bs) {$\left(P \setminus S\right) \cup \set{x}$};
\draw [-] (A) to (B);
\draw [-] (A) to (As);
\draw [-] (B) to (Bs);
\end{tikzpicture}
\qquad 
\begin{tikzpicture}[->,>=bad to,baseline=-3pt,node distance=1.3cm,thick,main node/.style={circle,draw,font=\Large,scale=0.5}]
\node at (0,0) [main node] (A) {$h$};
\node at (2,0) [main node] (B) {$g-h$};
\node at (0,-1)  (As) {$S\cup \set{x}$};
\node at (2,-1)  (Bs) {$P\setminus S $};
\draw [-] (A) to (B);
\draw [-] (A) to (As);
\draw [-] (B) to (Bs);
\end{tikzpicture}
    \caption{After identifying $\cbgp$ with $\overline{\mathcal M}_{g,P\cup \set{x}}$, the divisor class $C_{h,S}$ (resp.~$C_{h,S}^+$) corresponds to the stable $(P\cup \set{x})$-pointed graph depicted on the left (resp.~on the right).}
    \label{fig:chs}
\end{figure}
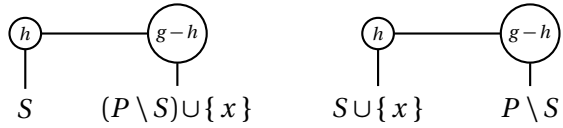

\subsection{Products of components on the universal curve}

In this section we compute the product of components $C_{h,S}$ in the Chow ring of the universal curve $\cbgp$.  This will motivate introducing the notation that appears in our main formula, Theorem~\ref{MainThm}. That notation will be first used in the following section.

Recall that by Convention \ref{convention_1isinS} every subset $S\subseteq P$ contains $1$.
We define a \emph{partial ordering on the stable bipartitions $(h,S)$} by setting 
\be\label{def:partial_order}
(h_1,S_1)\leq (h_2,S_2) \textrm{ if and only if } h_1\leq h_2 \textrm{ and } S_1\subseteq S_2.
\ee

\begin{notation} \label{firstnotation}
 For $r>0$ and a strictly ordered chain of stable bipartitions $(h_1,S_1) < \cdots < (h_r,S_r)$ and for nonnegative indices $i_1, \ldots, i_r$ and $j_1, \ldots, j_r$, we define the class
\[
C_{(h_1,\ldots,h_r),(S_1,\ldots,S_r)}^{(i_1,j_1),\ldots,(i_r,j_r)}\defeq  \;
 \begin{tikzpicture}[->,>=bad to,baseline=-3pt,node distance=1.3cm,thick,main node/.style={circle,draw,font=\Large,scale=0.5}]
\node at (0,0) [main node] (A) {$h_1$};
\node at (2,0) [main node] (B) {$h_2-h_1$};
\node at (4,0)  (C) {...};
\node at (6,0) [main node] (D) {$h_{r}-h_{r-1}$};
\node at (9,0) [main node] (E) {$g-h_r$};
\node at (0,-1.4)  (As) {$S_1$};
\node at (2,-1.4)  (Bs) {$S_2\setminus S_1$};
\node at (6,-1.4)  (Ds) {$S_r\setminus S_{r-1}$};
\node at (9,-1.4)  (Es) {$(P\setminus S_r) \cup \set{x}$};
\node at (.4,.5) (a) {\tiny $(i_1)$};
\node at (1.4,.5) (b1) {\tiny $(j_1)$};
\node at (2.6,.5) (b2) {\tiny $(i_2)$};
\node at (5.2,.5) (d1) {\tiny $(j_{r-1})$};
\node at (6.6,.5) (d2) {\tiny $(i_r)$};
\node at (8.4,.5) (e) {\tiny $(j_{r})$};
\draw [<<->>] (A) to (B);
\draw [<<-] (B) to (C);
\draw [->>] (C) to (D);
\draw [<<->>] (D) to (E);
\draw [-] (A) to (As);
\draw [-] (B) to (Bs);
\draw [-] (D) to (Ds);
\draw [-] (E) to (Es);
\end{tikzpicture}
\]
 in $R^\bullet(\M_{g,P\cup\{x\}}) = R^\bullet(\overline{\mathcal{C}}_{g,P})$.

With the same notation as above, we also define the classes
\[
X_{(h_1,\ldots,h_r),(S_1,\ldots,S_r)}^{(i_1,j_1),\ldots,(i_r,j_r)}\defeq  \;
 \begin{tikzpicture}[->,>=bad to,baseline=-3pt,node distance=1.3cm,thick,main node/.style={circle,draw,font=\Large,scale=0.5}]
\node at (0,0) [main node] (A) {$h_1$};
\node at (2,0) [main node] (B) {$h_2-h_1$};
\node at (4,0)  (C) {...};
\node at (6,0) [main node] (D) {$h_{r}-h_{r-1}$};
\node at (9,0) [main node] (E) {$g-h_r$};
\node at (0,-1.4)  (As) {$S_1$};
\node at (2,-1.4)  (Bs) {$S_2\setminus S_1$};
\node at (6,-1.4)  (Ds) {$S_r\setminus S_{r-1}$};
\node at (9,-1.4)  (Es) {$P\setminus S_r$};
\node at (.4,.5) (a) {\tiny $(i_1)$};
\node at (1.4,.5) (b1) {\tiny $(j_1)$};
\node at (2.6,.5) (b2) {\tiny $(i_2)$};
\node at (5.2,.5) (d1) {\tiny $(j_{r-1})$};
\node at (6.7,.5) (d2) {\tiny $(i_r)$};
\node at (8.4,.5) (e) {\tiny $(j_{r})$};
\draw [<<->>] (A) to (B);
\draw [<<-] (B) to (C);
\draw [->>] (C) to (D);
\draw [<<->>] (D) to (E);
\draw [-] (A) to (As);
\draw [-] (B) to (Bs);
\draw [-] (D) to (Ds);
\draw [-] (E) to (Es);
\end{tikzpicture}
\]
\[
Y_{(h_1,\ldots,h_r),(S_1,\ldots,S_r)}^{(i_1,j_1),\ldots,(i_{r+1},j_{r+1})}\defeq  \;
 \begin{tikzpicture}[->,>=bad to,baseline=-3pt,node distance=1.3cm,thick,main node/.style={circle,draw,font=\Large,scale=0.5}]
\node at (0,0) [main node] (A) {$h_1$};
\node at (2,0) [main node] (B) {$h_2-h_1$};
\node at (4,0)  (C) {...};
\node at (6,0) [main node] (D) {$h_{r}-h_{r-1}$};
\node at (9,0) [main node] (E) {$g-h_r-1$};
\node at (0,-1.4)  (As) {$S_1$};
\node at (2,-1.4)  (Bs) {$S_2\setminus S_1$};
\node at (6,-1.4)  (Ds) {$S_r\setminus S_{r-1}$};
\node at (9,-1.4)  (Es) {$P\setminus S_r$};
\node at (.4,.5) (a) {\tiny $(i_1)$};
\node at (1.4,.5) (b1) {\tiny $(j_1)$};
\node at (2.6,.5) (b2) {\tiny $(i_2)$};
\node at (5.2,.5) (d1) {\tiny $(j_{r-1})$};
\node at (6.7,.5) (d2) {\tiny $(i_r)$};
\node at (8.2,.5) (e) {\tiny $(j_{r})$};
\node at (9.6,.8) (e) {\tiny $(i_{r+1})$};
\node at (9.6,-.8) (e) {\tiny $(j_{r+1})$};
\draw [<<->>] (A) to (B);
\draw [<<-] (B) to (C);
\draw [->>] (C) to (D);
\draw [<<->>] (D) to (E);
\draw [<<->>] (E) to [out=30, in=-30,looseness=5] (E);
\draw [-] (A) to (As);
\draw [-] (B) to (Bs);
\draw [-] (D) to (Ds);
\draw [-] (E) to (Es);
\end{tikzpicture}
\]
\[
Z_{(h_1,\ldots,h_r),(S_1,\ldots,S_r)}^{(i_1,j_1),\ldots,(i_r,j_r),b}\defeq  \;
 \begin{tikzpicture}[->,>=bad to,baseline=-3pt,node distance=1.3cm,thick,main node/.style={circle,draw,font=\Large,scale=0.5}]
\node at (0,0) [main node] (A) {$h_1$};
\node at (2,0) [main node] (B) {$h_2-h_1$};
\node at (4,0)  (C) {...};
\node at (6,0) [main node] (D) {$h_{r}-h_{r-1}$};
\node at (9,0) [main node] (E) {$g-h_r$};
\node at (0,-1.4)  (As) {$S_1$};
\node at (2,-1.4)  (Bs) {$S_2\setminus S_1$};
\node at (6,-1.4)  (Ds) {$S_r\setminus S_{r-1}$};
\node at (9,-1.4)  (Es) {$P\setminus S_r$};
\node at (.4,.5) (a) {\tiny $(i_1)$};
\node at (1.4,.5) (b1) {\tiny $(j_1)$};
\node at (2.6,.5) (b2) {\tiny $(i_2)$};
\node at (5.2,.5) (d1) {\tiny $(j_{r-1})$};
\node at (6.7,.5) (d2) {\tiny $(i_r)$};
\node at (8.4,.5) (e) {\tiny $(j_{r})$};
\node at (9.4,.5) (e) {\footnotesize $\kappa_{b}$};
\draw [<<->>] (A) to (B);
\draw [<<-] (B) to (C);
\draw [->>] (C) to (D);
\draw [<<->>] (D) to (E);
\draw [-] (A) to (As);
\draw [-] (B) to (Bs);
\draw [-] (D) to (Ds);
\draw [-] (E) to (Es);
\end{tikzpicture}
\]
in $R^\bullet(\M_{g,P})$.
For later convenience, we allow $b \geq -1$ and we fix the convention that 
\be\label{Psi_Shift}
\kappa_{-1} \psi^t \defeq \psi^{t-1},\quad \psi^{-1}=0.
\ee
\end{notation}

The classes that appear in Theorem~\ref{MainThm} are those introduced in the previous notation, with a suitable choice of indices, and a suitable coefficient, as described in the following.
\begin{notation}\label{notation_tautological}
Let $r \geq 0$, $b\geq -1$ and $k_1,\ldots,k_r>0$ be integers, and let $(h_1,S_1)~<~\cdots~<~(h_r,S_r)$ be an ordered chain of stable bipartitions. Set $\mathbf{h}_r=(h_1,\ldots,h_r)$, $\mathbf{S}_r=(S_1,\ldots,S_r)$ and finally $\mathbf{k}_r~=~(k_1,\ldots,k_r)$. We define the codimension $\sum_{1\leq a \leq r} k_a$, resp. $b + \sum_{1\leq a \leq r} k_a$ classes:
\begin{align*}
\mathbf{X}_{(\mathbf{h}_r),(\mathbf{S}_r)}^{\mathbf{k}_r}
&\defeq  \begin{cases} 
1 &\textrm{ when } r=0,\\
\displaystyle \sum_{\substack{ 0\leq i_1 \leq k_1-1\\ \cdots \\ 0\leq i_r \leq k_r-1 }} \left(\prod_{j=1}^{r}(-1)^{k_j-1}\binom{k_j-1}{i_j}\right) X_{(h_1,\ldots,h_r),(S_1,\ldots,S_r)}^{(i_1,k_1-1-i_1),\ldots,(i_r,k_r-1-i_r)} & \textrm{ when } r>0,
\end{cases}
\\
\mathbf{Z}_{(\mathbf{h}_r),(\mathbf{S}_r)}^{\mathbf{k}_r,b}
&\defeq \begin{cases}
\kappa_b & \textrm{ when } r=0,\\
\displaystyle \sum_{\substack{ 0\leq i_1 \leq k_1-1\\ \cdots \\ 0\leq i_r \leq k_r-1 }}\left( \prod_{j=1}^{r}(-1)^{k_j-1}\binom{k_j-1}{i_j}\right) Z_{(h_1,\ldots,h_r),(S_1,\ldots,S_r)}^{(i_1,k_1-1-i_1),\ldots,(i_r,k_r-1-i_r),b} & \textrm{ when } r>0, \\  \displaystyle
\sum_{\substack{ 0\leq i_1 \leq k_1-1\\ \cdots \\ 0\leq i_r \leq k_r-2 }} \left(\prod_{j=1}^{r}(-1)^{k_j-1}\binom{k_j-1}{i_j}\right) X_{(h_1,\ldots,h_r),(S_1,\ldots,S_r)}^{(i_1,k_1-1-i_1),\ldots,(i_{r},k_{r}-2-i_{r})} & \textrm{when }b = -1.
\end{cases}
\end{align*}
Given additional integers $i$ and $j$, we define the codimension $i+j+1+\sum_{1\leq a \leq r} k_a $ class:
\begin{align*}
\widetilde{\mathbf{Y}}_{(\mathbf{h}_r),(\mathbf{S}_r)}^{\mathbf{k}_r,(i,j)}
&\defeq \begin{cases}
     \begin{tikzpicture}[->,>=bad to,baseline=-3pt,node distance=1.3cm,thick,main node/.style={circle,draw,font=\Large,scale=0.5}]
\node at (0,0) [main node] (A) {$g-1$};
\node at (.6,.5) (e) {\tiny $(i)$};
\node at (.6,-.5) (e) {\tiny $(j)$};
\draw [<<->>] (A) to [out=30, in=-30,looseness=5] (A);
\end{tikzpicture} & \textrm{ when } r=0,\\
\displaystyle
\sum_{\substack{ 0\leq i_1 \leq k_1-1\\ \cdots \\ 0\leq i_r \leq k_r-1 }}\left( \prod_{j=1}^{r}(-1)^{k_j-1}\binom{k_j-1}{i_j}\right) Y_{(h_1,\ldots,h_r),(S_1,\ldots,S_r)}^{(i_1,k_1-1-i_1),\ldots,(i_r,k_r-1-i_r),(i,j)} & \textrm{ when } r>0.
\end{cases}
\end{align*}
Finally, given a further stable bipartition $(h_{r+1},S_{r+1})>(h_r,S_r)$ as well, we define the codimension $i+j+1+\sum_{1\leq a \leq r} k_a $ class:
\begin{align*}
\widetilde{\mathbf{X}}_{(\mathbf{h}_r,h_{r+1}),(\mathbf{S}_r,S_{r+1})}^{\mathbf{k}_r,(i,j)}
&\defeq  \begin{cases}
\begin{tikzpicture}[->,>=bad to,baseline=-3pt,node distance=1.3cm,thick,main node/.style={circle,draw,font=\Large,scale=0.5}]
\node at (0,0) [main node] (A) {$h_1$};
\node at (2,0) [main node] (B) {$g-h_1$};
\node at (0,-1)  (As) {$S_1$};
\node at (2,-1)  (Bs) {$P\setminus S_1$};
\node at (.4,.4) (a) {\tiny $(i)$};
\node at (1.4,.4) (b1) {\tiny $(j)$};
\draw [<<->>] (A) to (B);
\draw [-] (A) to (As);
\draw [-] (B) to (Bs);
\end{tikzpicture} & \textrm{ when } r=0,\\
\displaystyle
\sum_{\substack{ 0\leq i_1 \leq k_1-1\\ \cdots \\ 0\leq i_r \leq k_r-1 }} \left(\prod_{j=1}^{r}(-1)^{k_j-1}\binom{k_j-1}{i_j}\right) X_{(h_1,\ldots,h_r,h_{r+1}),(S_1,\ldots,S_r,S_{r+1})}^{(i_1,k_1-1-i_1),\ldots,(i_r,k_r-1-i_r),(i,j)} & \textrm{ when } r>0.
\end{cases}
\end{align*}
For uniformity of notation in sums, it will be convenient to define the latter tautological class $\widetilde{\mathbf{X}}$ even when the index $r$ equals $-1$. In this case we set that class to equal zero. 
\end{notation}

The motivation for introducing the tautological classes described above will become  clear when in Section~\ref{formulazza} we will prove Theorem~\ref{MainThm}, but the reason to package the coefficients the way we did  already appears in the following lemma.

\begin{lemma}\label{lem:ChSprod}
 If $(h_1,S_1) \nleq (h_2,S_2)$ and $(h_2,S_2)\nleq (h_1,S_1)$ then \[C_{h_1,S_1}\cdot C_{h_2,S_2}=0.\] 
 Let $r>0$ and assume  $(h_1,S_1)< \cdots < (h_r,S_r)$ is a strictly ordered chain, and $k_1,\ldots,k_r>0$ are integers. Then we have 
 \[
\prod_{j=1}^rC_{h_j,S_j}^{k_j}=
 \sum_{\substack{ 0\leq i_1 \leq k_1-1\\ \cdots \\ 0\leq i_r \leq k_r-1 }}  \left( \prod_{j=1}^r (-1)^{ (k_j-1)} \binom{k_j-1}{i_j}\right) C_{(h_1,\ldots,h_r),(S_1,\ldots,S_r)}^{(i_1,k_1-1-i_1)\ldots(i_r,k_r-1-i_r)}.
 \]
\end{lemma}

\begin{proof}
This is a direct computation using \cite[Appendix]{graberpanda} where we use our convention that $S_i$ always contains the marking $1$ and $S^c_i\defeq \left( P \cup \set{x}\right) \setminus S_i$ always contains the extra marking $x$ coming from the identification $\overline{\mathcal{C}}_{g,P}=\mathcal{\overline{M}}_{g,P\cup\{x\}}$.

Let $G_{i}$  be the graph associated to $C_{h_i,S_i}$. The intersection $C_{h_1,S_1}\cdot C_{h_2,S_2}$ is the sum of all graphs $G$ with $2$ edges $e_1$ and $e_2$, such that contracting the edges $e_i$ gives the graph $G_i$. The genus of the vertex $v$ of $G$ with the marking $1$ has to equal $\min(h_1,h_2)$ and its markings have to be $S_1\cap S_2$. Since both the edges $e_1$ and $e_2$ separate the markings $1$ and $x$ only one of these edges can be incident to $v$. Contracting the edge $e_i$ \emph{not} incident to $v$ can only produce the graph $G_i$ associated to  $C_{h_i,S_i}$ if $h_i=\min(h_1,h_2)$ and $S_i=S_1\cap S_2$. This proves the first part of the Lemma.
    \begin{figure}[h] 
\begin{tikzpicture}[->,>=bad to,baseline=-3pt,node distance=1.3cm,thick,main node/.style={circle,draw,font=\Large,scale=0.5}]
\node at (0,0) [main node] (A) {$\min(h_1,h_2)$};
\node at (3,0) [main node] (B) {$\mid \! h_2-h_1 \!\mid $};
\node at (7,0) [main node] (C) {$g-\max(h_1,h_2)$};
\node at (0,-1.4)  (As) {$S_1\cap S_2$};
\node at (3,-1.4)  (Bs) {$(S_1\cup S_2)\setminus (S_1\cap S_2) $};
\node at (7,-1.4)  (Cs) {$(P\cup \{x\})\setminus (S_1\cup S_2)$};
\node at (-.5,.8) (v) {$v$};
\node at (5,.4) (e) {$e_i$};
\draw [-] (A) to (B);
\draw [-] (B) to (C);
\draw [-] (A) to (As);
\draw [-] (B) to (Bs);
\draw [-] (C) to (Cs);
\end{tikzpicture}
\caption{Graphs $G$ contracting generically to $G_1$ and $G_2$.}
\end{figure}
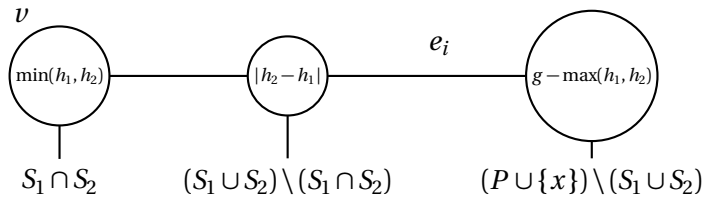

The second part of the statement follows from repeatedly applying the same procedure together with the fact that $C_{h_i,S_i}^{k_i}= \xi_{h_i,S_i,\,*}\left(-\psi_\bullet -\psi_\star\right)^{k_i-1}$, where $\bullet$ and $\star$ are the half edges associated to the edge of $C_{h_i,S_i}$.
\end{proof}

To conclude this section, we compute the pushforward of the classes of the previous lemma under the forgetful morphism $\pi\colon \, \M_{g,P\cup \{x\}}\to \M_{g,P}$. This will be key to Section~\ref{PFproddiv}.
\begin{lemma}\label{lemma:PF_C}
Let $r>0$ and assume $(h_1,S_1)< \cdots < (h_r,S_r)$ is a strictly ordered chain of stable bipartitions. Let $i_1,\ldots,i_r$, $j_1,\ldots,j_r$ be non-negative integers. We have
 \[
 \pi_*\left(C_{(h_1,\ldots,h_r),(S_1,\ldots,S_r)}^{(i_1,j_1),\ldots,(i_r,j_r)}\right)=
 \begin{cases}
 0 & j_r=0\\
 X_{(h_1,\ldots,h_r),(S_1,\ldots,S_r)}^{(i_1,j_1),\ldots,(i_r,j_r-1)}  & j_r>0.
 \end{cases}
 \]
\end{lemma}

\begin{proof}
This follows immediately from Lemma~\ref{lem:ChSprod} and from the String Equation (see \cite[Proposition~4.9]{ACG}).
\end{proof}

%%%%%%%%%%%%%%%%%%%%%%%%%%%%%%%%%%%%%%%%%%%%%%%%%%%%%%%%%%%%%%%%%%%%
\subsection{Push-forward of products of divisors on the universal curve}
\label{PFproddiv}

To establish our main result, Theorem~\ref{MainThm}, we will need a formula for the push-forward
\[
\pi_* \left( K_{\pi}^{\alpha} \cdot \prod_{p \in P} \sigma_p^{\beta_p}  \prod_{h,S}  C_{h,S}^{\gamma_{h,S}}\right)
\]
of an arbitrary product of divisor classes from the universal curve. By using the vanishing relations
\begin{enumerate}
    \item $\sigma_p\cdot \sigma_q=0$ for all $p\neq q$,
    \item $K_{\pi} \cdot \sigma_p = 0$ for all $p\in P$,
    \item $C_{h_1,S_1}\cdot C_{h_2,S_2} = 0$ if  $(h_1,S_1)\nleq (h_2,S_2)$ and  $(h_2,S_2)\nleq (h_1,S_1)$ (see Lemma \ref{lem:ChSprod}), 
\end{enumerate}
this is reduced to the problem of proving the remaining lemmas of this section. 

For the next result, we will also make use of the relation
\be \label{vanish} \sigma_p\cdot C_{h,S} = 0 \textrm{ for all } p\in S.\ee

\begin{lemma}\label{lemma:Sigma_cap_components}
Let $r\geq 0$  and assume  $(h_1,S_1)< \cdots <(h_r,S_r)$ is a strictly ordered chain of stable bipartitions of $(g,P)$. Let $b,k_1,\ldots,k_r>0$ be integers. Then
\[
\pi_\ast\left(\sigma_p^b\cdot \prod_{j=1}^r C_{h_j,S_j}^{k_j}\right) =
\begin{cases}
(-\psi_p)^{b-1}\mathbf{X}_{(\mathbf{h}_r),(\mathbf{S}_r)}^{\mathbf{k}_r} & \textrm{if }p\in P\setminus S_r \\
0 & \textrm{if }p\in S_r.
\end{cases}
\]
\end{lemma}

\begin{proof}
The second equality follows immediately from Equation~\eqref{vanish}. The first equality follows from Lemma~\ref{lem:ChSprod} and from the String Equation (see \cite[Proposition~4.9]{ACG}).
\end{proof}

\begin{remark}
The classes  $\psi_p^{b-1}\mathbf{X}_{(\mathbf{h}_r),(\mathbf{S}_r)}^{\mathbf{k}_r}$ belong to the set of standard generators because of the equality  $\psi_p^e\cdot \xi_{\Gamma\,\ast}(\alpha) = \xi_{\Gamma\,\ast}(\psi_p^e\cdot \alpha)$, which follows from the projection formula combined with the fact that psi classes pull back along the clutching morphisms.
\end{remark}

\begin{lemma}\label{lemma:K_cap_components} % K INTERSECTED WITH COMPONENTS
Let $r\geq 0$  and assume $(h_1,S_1)< \cdots < (h_r,S_r)$ is a strictly ordered chain of stable bipartitions of $(g,P)$. For all integers $k_1,\ldots,k_r>0$ and $b \geq 0$ we have the identity 
\[
\pi_\ast\left(K_\pi^b\cdot \prod_{j=1}^r C_{h_j,S_j}^{k_j}\right) = 
\mathbf{Z}_{(\mathbf{h}_r),(\mathbf{S}_r)}^{\mathbf{k}_r,b-1}.
\]
\end{lemma}

\begin{proof}
The case $b=0$ follows immediately from Lemma~\ref{lemma:PF_C} when $r>0$, and from the very definition of $\mathbf{Z}_{(\mathbf{h}_r),(\mathbf{S}_r)}^{\mathbf{k}_r,b-1}$ in the case $(r,b)=(0,0)$.

Let now $b>0$. Note that under the identification of the universal curve $\overline{\mathcal{C}}_{g,P}$ with $\overline{\mathcal{M}}_{g,P \cup \set{x}}$ the class $K_\pi$ corresponds to $\psi_x$. The claim then follows from Lemma~\ref{lem:ChSprod} and from the Dilaton Equation (see  \cite[Proposition~4.9]{ACG}).
\end{proof}

%%%%%%%%%%%%%%%%%%%%%%%%%%%%%%%%%%%%%%%%%%%%%%%%%%%%%%%%%%%%%%
\section{Proof of main theorem}\label{formulazza}

This section provides a proof of our main result, Theorem~\ref{MainThm}, using the notation established in Section~\ref{Notation}.  We prove the theorem by following Mumford (and later Chiodo), namely by applying the Grothendieck--Riemann--Roch formula to the universal curve $\pi$. There are, in principle, different ways to approach the calculation. Our approach is to reduce this computation to the pushforward along $\pi$ of products of divisors, and we know how to express them as linear combinations of decorated boundary strata classes following Section~\ref{PFproddiv}.

Consider the divisor class
\[
\mathsf D= \ell \widetilde{K}_{\pi} + \sum_{p\in P} d_p \sigma_p  + \sum_{h,S} a_{h,S} C_{h,S},
\]
on the universal curve $\cbgp$, where the indices $(h,S)$ run over the set of stable bipartitions of $(g,P)$ and $\ell$, $d_p$, $a_{h,S}\in \mathbb{Z}$. 
It will be convenient to write
\be\label{Decompose_D}
\mathsf D = \ell \widetilde K_\pi + \mathsf C + \mathsf S,
\ee
where
\[
\mathsf C = \sum_{h,S}a_{h,S}C_{h,S},\quad \mathsf S = \sum_{p\in P}d_p\sigma_p.
\]
For later use, we use the multinomial theorem and Lemma \ref{lem:ChSprod} to expand the power 
\begin{equation}\label{C_power}
    \begin{split}
        \frac{\mathsf C^a}{a!} &= \frac{1}{a!}\sum_{\substack{r \geq 0\\k_1+\cdots +k_r = a \\ k_j>0 \\ (h_1,S_1)< \cdots < (h_r,S_r)}}\binom{a}{k_1,\ldots,k_r}\prod_{j=1}^{r} a_{h_j,S_j}^{k_j}C_{h_j,S_j}^{k_j}\\
        &= \sum_{\substack{r \geq 0\\k_1+\cdots +k_r = a \\ k_j>0 \\ (h_1,S_1)< \cdots < (h_r,S_r)}}
        \sum_{\substack{ 0\leq i_1 \leq k_1-1\\ \cdots \\ 0\leq i_r \leq k_r-1 }}   \prod_{j=1}^r\left(\frac{a_{h_j,S_j}^{k_j}}{k_j!} (-1)^{ (k_j-1)} \binom{k_j-1}{i_j}\right) C_{(h_1,\ldots,h_r),(S_1,\ldots,S_r)}^{(i_1,k_1-1-i_1),\ldots,(i_r,k_r-1-i_r)},
    \end{split}      
\end{equation}
where  $(h_1,S_1)<\cdots<(h_r,S_r)$ denotes any strictly ordered chain of stable bipartitions (such partial order being defined in~\eqref{def:partial_order}).

\smallbreak
Let $\Sigma\subset \cbgp$
be the smooth closed codimension two substack parametrising the nodes in the fibers of the universal curve $\pi$.
Running the Grothendieck--Riemann--Roch formula we find
\[
\ch(R^\bullet\pi_\ast\O(\mathsf D)) 
= \pi_\ast(\ch(\O(\mathsf D))\cdot \Td^\vee(\Omega^1_\pi)) 
= \pi_\ast\left(e^{\mathsf D}\cdot \frac{\widetilde K_\pi}{e^{\widetilde K_\pi}-1}\cdot \Td^\vee(\O_\Sigma)^{-1}\right).
\]
A classical argument first given by Mumford and described in \cite[Chapter~17.5]{ACG} shows that $\Td^\vee(\O_\Sigma)^{-1}-1$  intersects $\widetilde K_\pi$ trivially. Therefore
\begin{align*}
    \ch(R^\bullet\pi_\ast\O(\mathsf D)) &=
    \pi_\ast\left(e^{\mathsf D}\cdot \frac{\widetilde K_\pi}{e^{\widetilde K_\pi}-1}\cdot \left(1+ \Td^\vee(\O_\Sigma)^{-1}-1\right)\right) \\
    &= \pi_\ast\left(e^{\mathsf D}\cdot \frac{\widetilde K_\pi}{e^{\widetilde K_\pi}-1}\right) + \pi_*\left(\left(\Td^\vee(\O_\Sigma)^{-1}-1\right)e^{\mathsf D}\right)\\
    &= \Omega + \Phi
\end{align*}
where $\Omega$ (resp.~$\Phi$) is defined to be the first summand (resp.~the second summand) of the previous equality.

The term $\Phi$ is computed in the following Lemma.
\begin{lemma}\label{lemma:todd_intersect_exp}
 We have 
 \begin{multline*}
 \pi_*\left(\left(\Td^\vee(\O_\Sigma)^{-1}-1\right)e^{\mathsf D}\right)=\\
    \sum_{t\geq 2} \sum_{\substack{a+b=t\\ b>0\textup{ even}\\ a\geq 0}} \frac{ B_{b}}{b!} \sum_{\substack{r \geq 0\\k_1+\cdots +k_r = a \\ k_j>0 \\ (h_1,S_1)< \cdots < (h_r,S_r)}} \prod_{j=1}^r  \frac{a_{h_j,S_j}^{k_j}}{k_j!}
   \sum_{ 0 \leq e \leq b-2} (-1)^{e}  \\
   \cdot \Bigg( \left(\sum_{(l,T)> (h_r,S_r)}\widetilde{\mathbf{X}}_{(\mathbf{h}_r,l),(\mathbf{S}_r,T)}^{\mathbf{k}_r,(e,b-2-e)}  \right) + \widetilde{\mathbf{Y}}_{(\mathbf{h}_r),(\mathbf{S}_r)}^{\mathbf{k}_r,(e,b-2-e)}
    +  (-1)^{k_r} \widetilde{\mathbf{X}}_{(\mathbf{h}_{r-1},h_r),(\mathbf{S}_{r-1},S_r)}^{\mathbf{k}_{r-1},(e+k_r,b-2-e)}  \Bigg).
\end{multline*}
\end{lemma}

Recall that in Notation~\ref{notation_tautological} we set the  class $\widetilde{\mathbf{X}}^{\mathbf k_{-1}}$ to equal zero.

\begin{proof}
A classical argument given in \cite[Chapter 17.5]{ACG} shows that $\widetilde{K}_\pi$ and $\sigma_p$ intersect $(\Td^\vee(\O_\Sigma)^{-1}-1)$ trivially. We therefore have
\begin{align*}
    \pi_*\left(\left(\Td^\vee(\O_\Sigma)^{-1}-1\right)e^{\mathsf D}\right) &=
    \pi_*\left(\left(\Td^\vee(\O_\Sigma)^{-1}-1\right)e^{\mathsf C}\right).
\end{align*}
The class $\Td^\vee(\O_\Sigma)^{-1}-1$ is also explicitly computed in \cite[Chapter 17.5]{ACG}  as 
\begin{equation}\label{eq:Psiup}
\Td^\vee(\O_\Sigma)^{-1}-1=\sum_{\substack{b>0 \\b \textup{ even}}} \frac{ B_{b}}{b!} 
\sum_{e=0}^{b-2} (-1)^{e} \left(\sum_{l,T} A_{l,T}^{(e,b-2-e)} + B^{(e,b-2-e)}\right)
\end{equation}
where $(l,T)$ runs over all stable bipartitions, and
\begin{center}
\begin{tabular}{c@{\hspace{1in}}c}
$A_{h,S}^{(i,j)}\defeq $\begin{tikzpicture}[->,>=bad to,baseline=-3pt,node distance=1.3cm,thick,main node/.style={circle,draw,font=\Large,scale=0.5}]
\node at (0,.4) [main node] (A) {$h$};
\node at (1,.4) [main node] (B) {$0$};
\node at (2,.4) [main node] (C) {$g-h$};
\node at (0,-.4)  (As) {$S$};
\node at (1,-.4)  (Bs) {$x $};
\node at (2,-.4)  (Cs) {$P\setminus S$};
\node at (.4,.8) (q) {\tiny $(i)$};
\node at (1.4,.8) (r) {\tiny $(j)$};
\draw [<<-] (A) to (B);
\draw [->>] (B) to (C);
\draw [-] (A) to (As);
\draw [-] (B) to (Bs);
\draw [-] (C) to (Cs);
\end{tikzpicture},
&  
$B^{(i,j)}\defeq $\begin{tikzpicture}[->,>=bad to,baseline=-3pt,node distance=1.3cm,thick,main node/.style={circle,draw,font=\Large,scale=0.5}]
\node at (0,0) [main node] (A) {$g-1$};
\node at (1.4,0) [main node]  (B) {$0$};
\node at (-1,0)  (As) {$P$};
\node at (2,0)  (Bs) {$x $};
\node at (.4,.6) (q) {\tiny $(i)$};
\node at (.4,-.6) (r) {\tiny $(j)$};
\draw [<<-] (A) to [out=40, in=140] (B);
\draw [<<-] (A) to [out=-40, in=-140] (B);
\draw [-] (A) to (As);
\draw [-] (B) to (Bs);
\end{tikzpicture}.
\end{tabular}
\end{center}

We expand $e^{\mathsf C}=\sum_{a\geq 0} \mathsf{C}^a/a!$ via \eqref{C_power}, so that multiplying \eqref{eq:Psiup} with $e^{\mathsf C}$ we obtain 
\begin{equation*}
\begin{split}
 \left(\Td^\vee(\O_\Sigma)^{-1}-1\right)e^{\mathsf C}=&\\
 \sum_{t\geq 1} \sum_{\substack{a+b=t\\ b>0 \textup{ even}}} \frac{ B_{b}}{b!}& \sum_{\substack{r \geq 0\\k_1+\cdots +k_r = a \\ k_j>0 \\ (h_1,S_1)< \cdots < (h_r,S_r)}}
   \sum_{\substack{ 0\leq i_1 \leq k_1-1\\ \cdots \\ 0\leq i_r \leq k_r-1 }} \prod_{j=1}^r \left(\frac{a_{h_j,S_j}^{k_j}}{k_j!} (-1)^{ (k_j-1)} \binom{k_j-1}{i_j}\right) \cdot\\
& \cdot \left(\sum_{e=0}^{b-2} (-1)^{e}  \left(\sum_{l,T} A_{l,T}^{(e,b-2-e)} + B^{(e,b-2-e)}\right)\right) C_{(h_1,\ldots,h_r),(S_1,\ldots,S_r)}^{(i_1,k_1-1-i_1)\ldots(i_r,k_r-1-i_r)}.
 \end{split}
\end{equation*}

By a straightforward computation in the spirit of Lemma \ref{lem:ChSprod} it follows that, if $(l,T)>(h_r,S_r)$ (or if $r=0$),
\begin{align*}
    \pi_*\left(A_{l,T}^{(i',j')} C_{(h_1,\ldots,h_r),(S_1,\ldots,S_r)}^{(i_1,j_1),\ldots,(i_r,j_r)}\right)&=\pi_*\left(
     \begin{tikzpicture}[->,>=bad to,baseline=-3pt,node distance=1.3cm,thick,main node/.style={circle,draw,font=\Large,scale=0.5},shift={(0,.5)}]
\node at (0,0) [main node] (A) {$h_1$};
\node at (1.8,0) [main node] (B) {$h_2-h_1$};
\node at (3.3,0)  (C) {...};
\node at (4.8,0) [main node] (D) {$l-h_{r}$};
\node at (6.5,0) [main node] (E) {$0$};
\node at (8.2,0) [main node] (F) {$g-l$};
\node at (0,-1.2)  (As) {$S_1$};
\node at (1.8,-1.2)  (Bs) {$S_2\setminus S_1$};
\node at (4.8,-1.2)  (Ds) {$T\setminus S_{r}$};
\node at (6.5,-1.2)  (Es) {$x$};
\node at (8.2,-1.2)  (Fs) {$P\setminus T$};
\node at (.4,.5) (a) {\tiny $(i_1)$};
\node at (1.2,.5) (b1) {\tiny $(j_1)$};
\node at (2.4,.5) (b2) {\tiny $(i_2)$};
\node at (4.2,.5) (d1) {\tiny $(j_{r})$};
\node at (5.4,.5) (d2) {\tiny $(i')$};
\node at (7.6,.5) (e) {\tiny $(j')$};
\draw [<<->>] (A) to (B);
\draw [<<-] (B) to (C);
\draw [->>] (C) to (D);
\draw [<<-] (D) to (E);
\draw [->>] (E) to (F);
\draw [-] (A) to (As);
\draw [-] (B) to (Bs);
\draw [-] (D) to (Ds);
\draw [-] (E) to (Es);
\draw [-] (F) to (Fs);
\end{tikzpicture}
\right)\\
&=X_{(h_1,\ldots,h_r,l),(S_1,\ldots,S_r,T)}^{(i_1,j_1),\ldots,(i_r,j_r)(i',j')}.
\end{align*}
If $(l,T)=(h_r,S_r)$ and $j_r=0$, 
\begin{align*}
    \pi_*\left(A_{l,T}^{(i',j')} C_{(h_1,\ldots, h_{r-1},l),(S_1,\ldots,S_{r-1}, T)}^{(i_1,j_1),\ldots,(i_{r-1},j_{r-1}),(i_r,0)}\right)&=-\pi_*\left(
     \begin{tikzpicture}[->,>=bad to,baseline=-3pt,node distance=1.3cm,thick,main node/.style={circle,draw,font=\Large,scale=0.5},shift={(0,.5)}]
\node at (0,0) [main node] (A) {$h_1$};
\node at (1.8,0) [main node] (B) {$h_2-h_1$};
\node at (3.3,0)  (C) {...};
\node at (4.8,0) [main node] (D) {$l-h_{r-1}$};
\node at (6.7,0) [main node] (E) {$0$};
\node at (8.2,0) [main node] (F) {$g-l$};
\node at (0,-1.2)  (As) {$S_1$};
\node at (1.8,-1.2)  (Bs) {$S_2\setminus S_1$};
\node at (4.8,-1.2)  (Ds) {$T\setminus S_{r-1}$};
\node at (6.7,-1.2)  (Es) {$x$};
\node at (8.2,-1.2)  (Fs) {$P\setminus T$};
\node at (.4,.5) (a) {\tiny $(i_1)$};
\node at (1.2,.5) (b1) {\tiny $(j_1)$};
\node at (2.4,.5) (b2) {\tiny $(i_2)$};
\node at (4,.5) (d1) {\scriptsize $(j_{r-1})$};
\node at (5.8,.5) (d2) {\tiny $(i_r+i'+1)$};
\node at (7.6,.5) (e) {\tiny $(j')$};
\draw [<<->>] (A) to (B);
\draw [<<-] (B) to (C);
\draw [->>] (C) to (D);
\draw [<<-] (D) to (E);
\draw [->>] (E) to (F);
\draw [-] (A) to (As);
\draw [-] (B) to (Bs);
\draw [-] (D) to (Ds);
\draw [-] (E) to (Es);
\draw [-] (F) to (Fs);
\end{tikzpicture}
\right)\\
&=-X_{(h_1,\ldots,h_{r-1},l),(S_1,\ldots,S_{r-1},T)}^{(i_1,j_1),\ldots,(i_{r-1}, j_{r-1})(i_r+i'+1,j')},
\end{align*}
and in all other cases,
\begin{align*}
    \pi_*\left(A_{l,T}^{(i',j')} C_{(h_1,\ldots,h_r),(S_1,\ldots,S_r)}^{(i_1,j_1),\ldots,(i_r,j_r)}\right)=0.
\end{align*}
Similarly
\begin{align*}
\pi_*\left(B^{(i',j')} C_{(h_1,\ldots,h_r),(S_1,\ldots,S_r)}^{(i_1,j_1),\ldots,(i_r,j_r)}\right)&=\pi_*\left(
     \begin{tikzpicture}[->,>=bad to,baseline=-3pt,node distance=1.3cm,thick,main node/.style={circle,draw,font=\Large,scale=0.5},shift={(0,.5)}]
\node at (0,0) [main node] (A) {$h_1$};
\node at (1.8,0) [main node] (B) {$h_2-h_1$};
\node at (3.3,0)  (C) {...};
\node at (4.8,0) [main node] (D) {$h_r-h_{r-1}$};
\node at (7,0) [main node] (E) {$g-h_r$};
\node at (8.4,0) [main node] (F) {$0$};
\node at (0,-1.2)  (As) {$S_1$};
\node at (1.8,-1.2)  (Bs) {$S_2\setminus S_1$};
\node at (4.8,-1.2)  (Ds) {$S_r\setminus S_{r-1}$};
\node at (7,-1.2)  (Es) {$P\setminus S_r$};
\node at (8.4,-1.2)  (Fs) {$x$};
\node at (.4,.5) (a) {\tiny $(i_1)$};
\node at (1.2,.5) (b1) {\tiny $(j_1)$};
\node at (2.4,.5) (b2) {\tiny $(i_2)$};
\node at (4,.5) (d1) {\tiny $(j_{r-1})$};
\node at (5.4,.5) (d2) {\tiny $(i_r)$};
\node at (6.4,.5) (e) {\tiny $(j_r)$};
\node at (7.4,.6) (q) {\tiny $(i')$};
\node at (7.4,-.6) (r) {\tiny $(j')$};
\draw [<<->>] (A) to (B);
\draw [<<-] (B) to (C);
\draw [->>] (C) to (D);
\draw [<<->>] (D) to (E);
\draw [<<-] (E) to [out=30, in=140] (F);
\draw [<<-] (E) to [out=-30, in=-140] (F);
\draw [-] (A) to (As);
\draw [-] (B) to (Bs);
\draw [-] (D) to (Ds);
\draw [-] (E) to (Es);
\draw [-] (F) to (Fs);
\end{tikzpicture}
\right)\\
&=Y_{(h_1,\ldots,h_r),(S_1,\ldots,S_r)}^{(i_1,j_1),\ldots,(i_r,j_r),(i',j')}.
\end{align*}

Putting everything together we deduce the statement.
\end{proof}

The remainder of this section is devoted to computing the remaining term
\[
\Omega = \pi_\ast \left(e^{\mathsf D}\cdot \frac{\widetilde K_\pi}{e^{\widetilde K_\pi}-1}\right).
\]
This will conclude the proof of Theorem \ref{MainThm}. First, in the notation of Equation \eqref{Decompose_D}, we find
\[
\Omega = \pi_\ast \left(e^{\mathsf C+\mathsf S}\cdot \sum_{t\geq 0}\frac{B_t(\ell)}{t!}\widetilde K_\pi^t \right)
\]
where we have used the identity
\[
\sum_{t\geq 0}\frac{B_t(\ell)}{t!}x^t = e^{\ell x}\frac{x}{e^x-1} 
\]
defining the Bernoulli polynomials $B_t(\ell)$.
Now we use that
\[
\widetilde K_\pi^t = K_\pi^t + (-1)^t\sum_{p\in P}\sigma_p^t,\quad \textrm{for all } t>0.
\]
We obtain
\begin{multline}\label{three_terms}
    \Omega = \pi_\ast \left(e^{\mathsf C+\mathsf S}\cdot \left(1+ \sum_{t>0}\frac{B_t(\ell)}{t!}K_\pi^t + \sum_{t>0}\frac{(-1)^t B_t(\ell)}{t!}\sum_{p\in P}\sigma_p^t\right)\right)\\
    =\pi_{\ast}e^{\mathsf C+\mathsf S} + \pi_*\left(e^{\mathsf C+\mathsf S}\cdot \left(-1+1+\sum_{t>0}\frac{B_t(\ell)}{t!}K_\pi^t \right)\right) \\
    +\pi_*\left(e^{\mathsf C+\mathsf S}\cdot \left(-1+1+\sum_{t>0}\frac{(-1)^t B_t(\ell)}{t!}\sum_{p\in P}\sigma_p^t\right)\right).
\end{multline}
Let us expand the second summand of \eqref{three_terms}.
Before the pushforward, we have
\[
e^{\mathsf C + \mathsf S}\cdot \left(1+\sum_{t>0}\frac{B_t(\ell)}{t!}K_\pi^t\right) = e^{\mathsf C}\cdot \left(e^{\mathsf S}+\sum_{t>0}\frac{B_t(\ell)}{t!}K_\pi^t\right)
\]
because $K_\pi\cdot \mathsf S = 0$. It follows that
\begin{align*}
    \pi_*\left(e^{\mathsf C+\mathsf S}\cdot \left(-1+1+\sum_{t>0}\frac{B_t(\ell)}{t!}K_\pi^t \right)\right) 
    &= -\pi_*e^{\mathsf C+\mathsf S}+\pi_*\left(e^{\mathsf C}\cdot \left(e^{\mathsf S}+\sum_{t>0}\frac{B_t(\ell)}{t!}K_\pi^t\right)\right)\\
    &=\pi_*\left(e^{\mathsf C}\cdot \left(-1+1+\sum_{t>0}\frac{B_t(\ell)}{t!}K_\pi^t\right)\right)\\
    &=-\pi_\ast e^{\mathsf C} + \pi_\ast\left(\left(\sum_{a\geq 0}\frac{\mathsf C^a}{a!}\right)\cdot \left(\sum_{b\geq 0}\frac{B_b(\ell)}{b!}K_\pi^b\right)\right)\\
    &=-\pi_\ast e^{\mathsf C} +\sum_{\substack{t>0\\a+b=t}} \frac{B_b(\ell)}{a!b!}\pi_\ast\left(\mathsf C^a\cdot K_\pi^b\right).
\end{align*}
It remains to compute the last summand in \eqref{three_terms}. We start by observing that the formula
\be\label{product_sections}
\mathsf S^\alpha \cdot \left(\sum_{p\in P}\sigma_p\right)^\beta
=\sum_{p\in P}d_p^\alpha\sigma_p^{\alpha+\beta}
\ee
holds whenever $(\alpha,\beta)\neq (0,0)$. We have
\begin{align*}
    e^{\mathsf S}\cdot \left(1+\sum_{t>0}\frac{(-1)^t B_t(\ell)}{t!}\sum_{p\in P}\sigma_p^t\right) &= 
    e^{\mathsf S}\cdot \left(1+\sum_{t>0}\frac{(-1)^t B_t(\ell)}{t!}\left(\sum_{p\in P}\sigma_p\right)^t\right)\\
    & = 1+\sum_{\substack{t>0\\\alpha+\beta=t}}\frac{(-1)^\beta B_\beta(\ell)}{\alpha!\beta!}\mathsf S^\alpha\cdot \left(\sum_{p\in P}\sigma_p\right)^\beta\\
    &=1+\sum_{\substack{t>0\\\alpha+\beta=t}}\frac{(-1)^\beta B_\beta(\ell)}{\alpha!\beta!}\sum_{p\in P}d_p^\alpha \sigma_p^{t}.
\end{align*}
We were allowed to apply \eqref{product_sections} in the last equality thanks to the fact that $\alpha$ and $\beta$ cannot both vanish. Now the last summand in \eqref{three_terms} equals
\[
-\pi_*e^{\mathsf C+\mathsf S}+\pi_*\left(e^{\mathsf C}\cdot\left(1+\sum_{\substack{t>0\\\alpha+\beta=t}}\frac{(-1)^\beta B_\beta(\ell)}{\alpha!\beta!}\sum_{p\in P}d_p^\alpha \sigma_p^{t} \right)\right).
\]
This can be rewritten as 
\begin{multline} \label{blah}
    -\pi_*e^{\mathsf C+\mathsf S}+\pi_*\left[
    1+\sum_{t>0}\left(\frac{\mathsf C^t}{t!} +
    \sum_{\substack{a+b=t \\ b>0}} \frac{\mathsf C^a}{a!}\sum_{\alpha+\beta = b}\frac{(-1)^\beta B_\beta(\ell)}{\alpha!\beta!}\sum_{p\in P}d_p^\alpha\sigma_p^{b}\right)\right] \\
    =-\pi_*e^{\mathsf C+\mathsf S}+\pi_*e^{\mathsf C}+\sum_{\substack{t>0 \\ a+b=t \\ b>0 \\ \alpha+\beta=b}}\frac{(-1)^\beta B_\beta(\ell)}{\alpha!\beta!a!}\sum_{p\in P}d_p^\alpha \pi_\ast \left(\mathsf C^a\cdot \sigma_p^b\right).
\end{multline}

Summing up, we obtain
\[
    \Omega =
    \sum_{\substack{t>0\\a+b=t}} \frac{B_b(\ell)}{a!b!}\pi_\ast\left(\mathsf C^a\cdot K^b\right)
+
\sum_{\substack{t>0\\ a+b=t \\ b>0 \\ \alpha+\beta = b}}\frac{(-1)^\beta B_\beta(\ell)}{\alpha!\beta!a!}\sum_{p\in P}d_p^\alpha \pi_*\left(\mathsf C^a\cdot \sigma_p^b\right).
\]

Combining \eqref{C_power} with Lemma \ref{lemma:K_cap_components}, the first summand of $\Omega$ reads
\[
\sum_{\substack{t>0\\a+b=t}} \frac{B_b(\ell)}{a!b!}\pi_\ast\left(\mathsf C^a\cdot K^b\right) = 
  \sum_{\substack{t>0\\a+b=t}} \frac{B_b(\ell)}{b!}\sum_{\substack{r\geq 0\\k_1+\cdots +k_r = a \\ k_j>0 \\ (h_1,S_1)< \cdots < (h_r,S_r)}}\left(\prod_{j=1}^{r} \frac{a_{h_j,S_j}^{k_j}}{k_j!}\right) \mathbf{Z}_{(\mathbf{h}_r),(\mathbf{S}_r)}^{\mathbf{k}_r,b-1}.   
\]
By Lemma \ref{lemma:Sigma_cap_components}, the second summand of $\Omega$ reads
\begin{multline*}
\sum_{\substack{t>0\\ a+b=t \\ b>0 \\ \alpha+\beta = b}}\frac{(-1)^\beta B_\beta(\ell)}{\alpha!\beta!a!}\sum_{p\in P}d_p^\alpha \pi_*\left(\mathsf C^a\cdot \sigma_p^b\right) = 
\\
\sum_{\substack{t>0\\ a+b=t \\ b>0 \\ \alpha+\beta = b}}\frac{(-1)^\beta B_\beta(\ell)}{\alpha!\beta!}\sum_{\substack{r\geq 0\\k_1+\cdots +k_r = a \\ k_j>0 \\ (h_1,S_1)< \cdots < (h_r,S_r)}}\left(\prod_{j=1}^{r} \frac{a_{h_j,S_j}^{k_j}}{k_j!} \sum_{p\in P\setminus S_r} d_p^\alpha(-\psi_p)^{b-1}\right)\mathbf{X}_{(\mathbf{h}_r),(\mathbf{S}_r)}^{\mathbf{k}_r}. 
\end{multline*}
This  concludes the proof of Theorem~\ref{MainThm}.

\begin{example} \label{esempio0}
As a sanity check, we compute $\ch_0 (R^\bullet\pi_\ast\O(\mathsf D))$ using our formula in Theorem~\ref{MainThm}. This means extracting the term with degree equal to $1$ in the variable $t$. In particular, $\Phi$ does not contribute.

The only nonzero contribution from $\Omega$ occurs when $a=0$ and $b=1$ and it equals (first summand)
\[
\left(\ell- \frac{1}{2}\right) \kappa_0=\left(\ell-\frac{1}{2}\right)(2g-2+n)
\]
plus (second summand)
\[
\sum_{p \in P} d_p- \left(\ell-\frac{1}{2}\right) n
\]
which gives, for $d \defeq \ell(2g-2)+\sum_{p \in P} d_p$, the Riemann--Roch formula
\[
\ch_0 (R^\bullet\pi_\ast\O(\mathsf D))= d+1-g.
\]
\end{example}

\begin{example} \label{esempio}
Let us compute $\ch_1 (R^\bullet\pi_\ast\O(\mathsf D))$ in the generating set (which is actually a basis as long as $g \geq 3$) for the rational Chow group of codimension-$1$ classes of $\overline{\mathcal{M}}_{g,P}$ consisting of 
\[
\kappa_1, \set{\psi_p}_{p\in P}, \delta_{\text{irr}}, \set{\delta_{h,S}}_{(h,S)}.
\] 
This amounts to extracting the term of degree  $2$ in the variable $t$ from the formula of Theorem \ref{MainThm}.

The summand $\Phi$ only contributes to $\delta=\delta_{\text{irr}}+ \sum_{h,S} \delta_{h,S}$, and with coefficient $\frac{1}{12}$.

The summand $\Omega$ contributes to $\kappa_1$ with coefficient (from $a=0$ and $b=2$)
\[
\frac{B_2(\ell)}{2!} = \frac{\ell^2-\ell+\frac{1}{6}}{2}.
\]
It contributes to $\psi_p$ for $p \in P$ with coefficient (also from $a=0$ and $b=2$ but from the second summand of $\Omega$)
\[
-\frac{1}{2} d_p^2 + \left(\ell-\frac{1}{2}\right) d_p - \frac{\ell^2-\ell+\frac{1}{6}}{2}.
\]
(The three summands correspond to the cases $(\alpha,\beta)=(2,0)$,  $(\alpha,\beta)=(1,1)$, and $(\alpha,\beta)=(0,2)$ respectively).

Furthermore, the term $\Omega$  contributes to $\delta_{h,S}$ as follows. Setting $d_{S^c} \defeq \sum_{p \in P \setminus S} d_p$ the contribution of $\Omega$ with $a=b=1$ reads:
\[
\left(\left(\ell-\frac{1}{2}\right) (2g-2h -1)+ d_{S^c}\right)\cdot  a_{h,S}.
\]
(We get $B_1(\ell)\cdot a_{h,S}\cdot \mathbf{Z}^{{k}_1=1,0}_{h,S} = (\ell-1/2)\cdot a_{h,S}\cdot (2g-2h-1+|S^c|)$ from the first summand of $\Omega$. A further contribution $d_{S^c}\cdot a_{h,S}$ comes from $(\alpha,\beta) = (1,0)$, whereas $(\alpha,\beta) = (0,1)$ contributes $-(\ell-1/2)\cdot a_{h,S}|S^c|$.)

Finally, the contribution of $\Omega$ with $(a,b)=(2,0)$ is
\[
- \frac{a_{h,S}^2}{2}.
\]

The coefficient of $\delta_{h,S}$ is therefore
\[
\frac{1}{12}+\frac{1}{2}a_{h,S} \cdot \left((2g-2h-1)(2\ell-1)+2d_{S^c}-a_{h,S}\right).
\]
\end{example}

%%%%%%%%%%%%%%%%%%%%%%%%%%%%%%%%%%%%%%%%%%%%%%%%%%%%%%%%%%%%
\section{Pullback of Brill--Noether classes via Abel--Jacobi sections}\label{Jacobians}

In this section we review the definition of compactified universal Jacobians $\jbgp(\phi)$ and then define the cohomological universal Brill--Noether classes 
\[
\mathsf w^r_d(\phi)\in A^{g-\rho} (\jbgp(\phi)),
\] 
where $\rho=g-(r+1)(g-d+r)$ is the Brill--Noether number. We always assume $r\geq 0$ and $d< g+r$ throughout. 

For fixed integers $\ell$ and $d_P\defeq \set{d_p|p \in P}$, in \eqref{fancypullback} we define the pullbacks 
\[
\mathsf Z^r_{\ell, d_P}(\phi)=\mathsf s^\ast \mathsf w^r_d(\phi),
\]
where $\mathsf s= \mathsf s_{\ell, d_P}$ is the Abel--Jacobi section defined by~\eqref{aj}. 

Finally, we  observe how the main result of the previous section allows one to explicitly compute the classes $\mathsf Z^r_{\ell, d_P}(\phi)$ in terms of decorated boundary strata classes, for all  $\phi$'s such that the section $\mathsf s$ is a well-defined morphism on $\mbgp$.

%%%%%%%%%%%%%%%%%%%%%%%%%%%%%%%%%%%%%%%%%%%%%%%%%%%%%%%%%%%%%%%%%%%%
\subsection{Compactified universal Jacobians}

We first review the definition of the stability space $V_{g,P}^d$ from \cite[Definition~3.2]{KassPagani1} and the notion of nondegenerate elements therein.  An element 
\[
\phi \in V_{g,P}^d
\]
is an assignment, for every  stable $P$-pointed curve $(C,P)$ of genus $g$ and every irreducible component $C' \subseteq C$, of a real number $\phi(C, P)_{C'}$  such that 
\[ 
\sum_{C' \subseteq C} \phi(C, P)_{C'} = d,
\] 
and 
\begin{enumerate}
\item if $\alpha \colon \, (C, P) \to (D, Q)$ is a homeomorphism of pointed curves, then 
\[
\phi(D,Q)= \phi(\alpha(C, P));
\] 
\item the assignment $\phi$ is compatible with degenerations of pointed curves (in the sense of \cite[Definition~3.2]{KassPagani1}).
\end{enumerate}

The notion of $\phi$-(semi)stability  was introduced in \cite[Definitions~4.1 and~4.2]{KassPagani1}:

\begin{definition} \label{semistab}
Given $\phi \in V_{g,P}^d$ we say that a family $F$ of rank~$1$ torsion free sheaves of degree~$d$ on a family of stable curves is \emph{$\phi$-stable} if the inequality
\begin{equation} \label{eqnsemistab}
		\left| \deg_{C_0}(F)- \sum \limits_{C' \subseteq C_0} \phi(C,P)_{C'} + \frac{\delta_{C_0}(F)}{2} \right| <  \frac{\#(C_0 \cap \overline{C_0^{c}})-\delta_{C_0}(F)}{2}
	\end{equation}
holds for every stable $P$-pointed curve $(C, P)$ of genus $g$ of the family, and for every subcurve (i.e.~a union of irreducible components) $\emptyset \neq C_0 \subsetneq C$. Here $\delta_{C_{0}}(F)$ denotes the number of nodes $p \in C_0 \cap \overline{C_0^{c}}$ such that the stalk of $F$ at $p$ fails to be locally free. Semistability with respect to $\phi$ is defined by allowing equality in \eqref{eqnsemistab}.

A stability parameter $\phi \in V_{g,P}^d$ is said to be \emph{nondegenerate} when $\phi$-semistability coincides with $\phi$-stability for all stable $P$-pointed curves of genus $g$. 
\end{definition}

For all $\phi \in V_{g,P}^d$ there exists a moduli stack $\jbgp(\phi)$ of $\phi$-semistable sheaves on stable curves, which comes with a forgetful morphism \[p\colon \, \jbgp(\phi) \to\mbgp.\] When $\phi$ is nondegenerate, by \cite[Corollary~4.4]{KassPagani1}, the stack $\jbgp(\phi)$ is a smooth Deligne--Mumford stack, and the morphism $p$ is representable, proper and flat.

%%%%%%%%%%%%%%%%%%%%%%%%%%%%%%%%%%%%%%%%%%%%%%%%%%%%%%%%%%%%%%%%%
\subsection{Universal Brill--Noether classes and their pullbacks} \label{PBBN}

Let $\phi \in V_{g,P}^d$ be nondegenerate. Then by \cite[Corollary~4.3]{KassPagani0} and \cite[Lemma~3.35]{KassPagani1} combined with our general assumption $P \neq \emptyset$, there exists a tautological family $\mathscr L (\phi)$ of rank $1$ torsion free sheaves of degree $d$ on the total space of the universal curve 
\[
\widetilde{\pi} \colon \, \jbgp(\phi) \times_{\mbgp} \cbgp \to \jbgp(\phi).
\]
Recall the following notation from \cite[Ch.~14]{Ful}. Let $c = \sum_{k\in \Z}c_k$ be a formal sum of elements in a ring $R$. Then the $p\times p$ determinant $|c_{q+j-i}|$ in $R$ is denoted
\[
\Delta^{(p)}_{q}c =
\begin{vmatrix}
c_q & c_{q+1} & \cdots & c_{q+p-1} \\
c_{q-1} & c_q & \cdots & c_{q+p-2} \\
\vdots & \vdots & \ddots & \vdots \\
c_{q-p+1} & c_{q-p+2} & \cdots & c_q
\end{vmatrix}.
\]
Generalising the idea of \cite[Definition~3.38]{KassPagani0} (where the authors extended the universal theta divisor $\mathsf w^0_{g-1}$), we define the (universal, cohomological) \emph{Brill--Noether class} using the Thom--Porteous formula, namely by
\begin{equation} \label{wrdphi}
\mathsf w_d^r(\phi) \defeq \Delta_{g-d+r}^{(r+1)}c (-R^\bullet\widetilde{\pi}_\ast \mathscr L(\phi)) \in A^{g-\rho}(\jbgp(\phi)),
\end{equation}
for $\rho = g- (r+1)(g-d+r)$ the Brill--Noether number. One can interpret the class \eqref{wrdphi} as follows. Define the \emph{universal Brill--Noether scheme} as the closed subscheme
\be\label{Wrd43}
\mathcal W^r_d(\phi) = \Fit_{g-d+r}(R^1\widetilde{\pi}_\ast \mathscr L(\phi)) \subset \jbgp(\phi),
\ee
defined by the $(g-d+r)$-th Fitting ideal of $R^1\widetilde{\pi}_\ast \mathscr L(\phi)$ (see \cite[Ch.~21]{ACG} for the use of Fitting ideals in Brill--Noether theory). Then the Poincar\'e dual of \eqref{wrdphi} is the class that $\mathcal W^r_d(\phi)$ would have as its fundamental class if it were pure of the expected codimension $g-\rho$. The scheme~\eqref{Wrd43} has an explicit description as a degeneracy scheme, which was already described in the proof of \cite[Lemma~6]{HKP} in the case $r=d=0$.
Fix a sufficiently $\widetilde{\pi}$-ample divisor $H$, and consider the short exact sequence
\[
0\ra \mathscr L(\phi)\ra \mathscr L(\phi)(H) \xrightarrow{u} \mathscr L(\phi)\otimes \O_H(H)\ra 0.
\]
Pushing this forward via $\widetilde{\pi}$ yields a presentation
\[
\mathscr E_0\,\, \overset{\widetilde{\pi}_\ast u}{\lra}\,\, \mathscr E_1 \ra R^1\widetilde{\pi}_\ast \mathscr L(\phi) \ra 0
\]
of $R^1\widetilde{\pi}_\ast \mathscr L(\phi)$, where $\widetilde{\pi}_\ast u$ is a morphism of vector bundles whose virtual rank is
\[
\rk \mathscr E_0 - \rk \mathscr E_1 = d-g+1
\]
by Riemann--Roch. The $k$-th degeneracy scheme of $\widetilde{\pi}_\ast u$, where $k=\rk \mathscr E_0-(r+1) = \rk \mathscr E_1-(g-d+r)$, is by definition the zero scheme
\be\label{VanLocus}
Z\left(\wedge^{k+1} \widetilde{\pi}_\ast u\right)\subset \jbgp(\phi),
\ee
which agrees with \eqref{Wrd43} by the general theory of Fitting ideals. Note that, by this identification, the vanishing locus \eqref{VanLocus} is independent of the choice of $H$. Moreover, $\mathcal W^r_d(\phi)$ is set-theoretically supported on
\[
\Set{(C, P, F) | h^0(C,F)>r} \subset \jbgp(\phi).
\]
The definition \eqref{wrdphi} is motivated by the following lemma. 

\begin{lemma} \label{makessense} 
The class $\mathsf w^r_d(\phi)$ is supported on $\mathcal W^r_d(\phi)$.
If the Brill--Noether scheme $\mathcal W^r_d(\phi)$ is pure of the expected codimension $g-\rho$, then $\mathsf w^r_d(\phi)$ is its fundamental class.
\end{lemma}

\begin{proof}
The first statement is proven in exactly the same manner as the first statement of \cite[Lemma~6]{HKP} (dealing with the case $r=d=0$), namely by observing that the class \eqref{wrdphi} is by construction supported on the degeneracy scheme \eqref{VanLocus}.
The second statement follows from \cite[Theorem~14.4]{Ful}. 
\end{proof}

\begin{example} \label{drc1}
For $r=0$ we have 
\be\label{Rzero}
\mathsf w^0_d(\phi)= c_{g-d} (-R^\bullet\widetilde{\pi}_\ast \mathscr{L}(\phi)).
\ee
These classes can therefore be seen as some formal analogues of the $\lambda$-classes on $\mbgp$, where $-R^\bullet\widetilde{\pi}_\ast \mathscr{L}(\phi)$ is taking the role of the pushforward of the relative dualising sheaf, namely of the Hodge bundle $\mathbb E=\pi_\ast\omega_\pi$. Note that for fixed $d$ the classes \eqref{Rzero} determine, by their defining formula \eqref{wrdphi}, all other classes $\mathsf w^r_d(\phi)$ for arbitrary $r$.
\end{example}

\begin{remark}
While the restriction $\mathcal W^0_d$ of $\mathcal W^0_d(\phi)$ to $\mathcal{M}_{g,P}$ always has the expected dimension (being the image of the $d$-th symmetric product of the universal curve under the summation map), arguing as in \cite[Remark~7]{HKP} one sees that for each stable bipartition $(h,S)$ there exists a nondegenerate $\phi$ such that $\mathcal W^0_d(\phi)$ contains the inverse image in $\jbgp(\phi)$ of the boundary divisor $\Delta_{h,S}$. In particular, $\mathcal W^0_d(\phi)$ is, in general, not even equidimensional.
\end{remark}

From now on we fix integers $\ell\in \mathbb Z$ and  $d_P\defeq \set{d_p|p \in P}$ and set $d \defeq \ell(2g-2)+\sum_pd_p$. For $\phi \in V_{g,n}^d$ nondegenerate, we define the rational map 
\be\label{AJsection73}
\mathsf s = \mathsf s_{\ell, d_P}(\phi) \colon \, \mbgp \dashrightarrow \jbgp(\phi)
\ee
by Rule~\eqref{aj}, for some choice of coefficients $a_{h,S}$. (This map is actually independent of the coefficients $a_{h,S}$ of $C_{h,S}$ as these divisors are zero on the open dense substack that parametrises line bundles over smooth pointed curves). 
\begin{definition} \label{fancydef} We define the pullback classes $\mathsf Z_{\ell, d_P}^r(\phi)$  by the formula
\begin{equation} \label{fancypullback} 
\mathsf Z_{\ell, d_P}^r(\phi) \defeq \mathsf s^*\mathsf w^r_d(\phi) = {p}_*\left(  \mathsf w^r_d(\phi) \cdot \left[\overline{\Sigma}(\phi)\right]  \right),
\end{equation}
where $\overline{\Sigma}(\phi)$ is the closure in $\jbgp(\phi)$ of the image of the section $\mathsf s$.
\end{definition} The second equality of Formula~\eqref{fancypullback} follows from the definition of pullback of an algebraic class by a rational map, and it is well-defined because $\jbgp(\phi)$ is proper.

When $\phi$ is such that the line bundle $\mathsf D$ of \eqref{univdivisor} is $\phi$-stable, the map \eqref{AJsection73} is a well-defined morphism on $\mbgp$ but, because the map is insensitive to the coefficients $a_{h,S}$, the converse is not true. 
\begin{definition} We define the open substack
\[
U(\phi)\defeq U_{\ell, d_P}(\phi)\subset \mbgp
\] 
to be the largest locus where the Abel--Jacobi section $\mathsf{s}= \mathsf{s}_{\ell, d_P}(\phi)$ extends to a well-defined morphism.
\end{definition}
In \cite[Section~6.1]{KassPagani1} the authors describe the locus $U(\phi)$ in terms of $\mathsf{D}$, and we now review that description. For all nondegenerate $\phi \in V_{g,P}^d$ there is a unique modification $\mathsf{D}(\phi)$ of $\mathsf D$ that coincides with $\mathsf D$ on the locus parametrising smooth curves and that is $\phi$-stable on all curves with exactly $1$ node. More explicitly, $\mathsf{D}(\phi)$ is obtained from $\mathsf D$ by modifying the coefficients $a_{h,S}$ of $C_{h,S}$ into coefficients $a_{h,S}(\phi)$ in the unique way that makes the resulting $\mathsf{D}(\phi)$ a divisor that is $\phi$-stable on all curves of $\mbgp$ with $1$ node. By \cite[Proposition~6.4]{KassPagani1} the  open substack  
$U(\phi)$
 can be characterised as the locus of $\mbgp$ where $\mathsf{D}(\phi)=\mathsf{D}_{\ell, d_P}(\phi) $ is $\phi$-stable. 

We now show how Theorem~\ref{MainThm} allows one to compute the restriction to $U(\phi)$ of the class $\mathsf s^*\mathsf w^r_d(\phi)$. Chiodo's formula recovers the particular case when $\mathsf{D}(\phi)$ equals $\ell \widetilde{K}_{\pi} +\sum_{p \in P} d_p \sigma_p$.

\begin{corollary}\label{cor:PB} Let $\phi\in V_{g,P}^d$ be nondegenerate. Then the equality of classes
\be \label{chernclasses}
\mathsf Z_{\ell, d_P}^r(\phi) = \Delta_{g-d+r}^{(r+1)} c (-R^\bullet\pi_\ast\mathscr O(\mathsf{D}(\phi)) )
\ee
holds on the open substack $U(\phi)$ of $\mbgp$.% where $\mathsf{D}(\phi)$ is $\phi$-stable.
\end{corollary}
\begin{proof} 
Consider the Cartesian square
\[
\begin{tikzcd}
\cbgp\MySymb{dr} \arrow[dashed]{r}{{\mathsf s'}}\arrow{d}{\pi} & \jbgp(\phi) \times_{\mbgp}\cbgp \arrow{d}{\widetilde{\pi}} \\
\mbgp \arrow[dashed]{r}{\mathsf s} & \jbgp(\phi)
\end{tikzcd}
\]
defining ${\mathsf s}'$.
We have the following equalities in the Chow group of $U(\phi)$:
\[
\mathsf s^*c_k(-R^\bullet\widetilde{\pi}_\ast \mathscr{L}(\phi))=c_k \mathsf s^*(-R^\bullet\widetilde{\pi}_\ast \mathscr{L}(\phi)) = c_k (-R^\bullet\pi_\ast {\mathsf s'}^* \mathscr{L}(\phi)) = c_k (-R^\bullet\pi_\ast\O(\mathsf{D}(\phi))).
\]
All equalities require to restrict to the locus where $\mathsf s$ is a morphism. The first follows from the fact that Chern classes commute with pullbacks via lci morphisms. The second is cohomology and base change \cite[Theorem~8.3.2]{illusie}, using that $\widetilde{\pi}$ is flat and $R^\bullet\widetilde{\pi}_\ast \mathscr{L}(\phi)$ is represented by a two-term complex of vector bundles. The third and the last follow from the definition of a tautological sheaf and of ${\mathsf s'}$.
Formula~\eqref{chernclasses} now follows from the definition of $\mathsf Z_{\ell, d_P}^r(\phi)$ and from the fact that the pullback along the morphism $\mathsf s$ is a ring homomorphism.
\end{proof}

Combining Formula~\eqref{chernclasses} with the formula
\be \label{inversion}
c_t(\mathsf F)=\left[ \exp \left( \sum_{s \geq 1} (-1)^{s-1}(s-1)! \ch_s(\mathsf F) \right)\right]_{t} 
\ee 
expressing the Chern classes of a $K$-theory element $\mathsf F$ in terms of the Chern character, and then applying Theorem~\ref{MainThm}, yields an explicit formula, in terms of decorated boundary strata classes, for the restriction of $\mathsf Z_{\ell, d_P}^r(\phi)$ to the open locus $U(\phi)$ of $\mbgp$. In particular, this computes $\mathsf Z_{\ell, d_P}^r(\phi)$ for all $\phi$ such that the corresponding Abel--Jacobi section~\eqref{AJsection73} extends to a morphism on $\mbgp$.

%%%%%%%%%%%%%%%%%%%%%%%%%%%%%%%%%%%%%%%%%%%%%%%%%%%%%%%%%%%%%%%%%%%%
\subsection{Relation to the Double Ramification Cycle} \label{drc2}

We conclude this section by relating the classes $\mathsf Z^r_{\ell, d_P}(\phi)$ (defined in \ref{fancydef}) for $r=d=0$ to the large body of literature on the  Double Ramification Cycle (DRC). We will start by introducing the DRC, following the perspective of \cite{HKP}, which is in turn based on the resolution of the indeterminacy of the Abel--Jacobi section by  D.~Holmes \cite{holmes}. For more details we refer the reader to \cite{HKP}.

Let  $\mathcal{J}_{g,P}^{\underline{0}}$ be the universal generalised Jacobian, or the moduli stack of multidegree zero line bundles on stable curves (equivalently, the unique semiabelian extension of the degree zero universal Jacobian over $\mathcal{M}_{g,P}$). For fixed integers $\ell \in \mathbb{Z}, d_P \colon \, P \to \mathbb{Z}$ such that 
\[
\ell(2g-2) + \sum_{i\in P}~d_i=0,
\]
let $\mathsf S \subset \mathcal{J}_{g,P}^{\underline{0}}$ be the closure of the image of the Abel--Jacobi section $\mathsf s = \mathsf s_{\ell, d_P}\colon \, \mbgp \dashrightarrow \mathcal{J}_{g,P}^{\underline{0}}$. Call $f$ the restriction to $\mathsf S$ of the forgetful morphism $\mathcal{J}_{g,P}^{\underline{0}} \to \mbgp$, and consider the  fiber product diagram
\[
\begin{tikzcd}[row sep = large]
\mathsf S \times_{\mbgp} \mathcal{J}_{g,P}^{\underline{0}} \MySymb{dr} \arrow[]{r}\arrow{d} &  \mathcal{J}_{g,P}^{\underline{0}} \arrow{d} \\
\mathsf S \arrow[bend left,swap]{u}[description]{\tilde{e}} \arrow[bend right]{u}[description]{\tilde{\mathsf s}} \arrow{r}{f} &  \mbgp \arrow[bend left,swap]{u}[description]{e} \arrow[bend right,dashed]{u}[description]{\mathsf s} 
\end{tikzcd}
\]
that defines the upper left corner. Here $\tilde{\mathsf s}$ is the inclusion and $\tilde{e}$ is the pullback of the zero section $e$. Denoting by $[\widetilde{\mathsf S}]$ the class of the image of $\tilde{\mathsf s}$,  one can define the $\ell$-twisted DRC following Holmes' work \cite{holmes} by
\begin{equation} \label{defDRC}
\mathsf{DRC}(\ell, d_P) \defeq f_* \tilde{e}^* [\widetilde{\mathsf S}].
\end{equation}
(The fact that when $\ell=0$ this definition coincides with the ``usual'' DRC  defined as the pushforward of the virtual class on the moduli space of relative stable maps to rubber $\mathbb{P}^1$ follows from \cite[Theorem~1.3]{holmes}, combined with the observation in \cite[Lemma~11]{HKP} that Holmes' stack $\mathcal{M}_{g,n}^{\lozenge}$ equals the normalisation of $\mathsf S$).

Denoting by $[\mathsf E]$ the class of the image of the zero section in $\mathcal{J}_{g,P}^{\underline{0}}$,  we deduce the equality of classes
\be
\mathsf{DRC}(\ell, d_P) = \mathsf{s}_{\ell, d_P}^*[\mathsf E]
\ee
by the push-pull formula and by the definition of pullback along the rational map $\mathsf s$.
This expression for the DRC is now closely related to the definition of the classes $\mathsf Z$ (Definition~\ref{fancydef}). Indeed, whenever $\phi \in V_{g,n}^0$ is nondegenerate and such that  $\mathcal{J}_{g,P}^{\underline{0}}\subset \jbgp(\phi)$, by \cite[Corollary~10]{HKP} we have that $\mathsf{w}^0_0(\phi)= [\mathsf E]$, so that $\mathsf{DRC}(\ell, d_P)=\mathsf{Z}^0_{\ell, d_P}(\phi)$. Combining this with Corollary~\ref{cor:PB}, we deduce the equality
\be \label{overct}
\mathsf{DRC}(\ell, d_P)_{|U(\phi)} = c_g (-R^\bullet{\pi}_\ast \mathsf{D}(\phi))_{|U(\phi)},
\ee
which is valid whenever the inclusion  $\mathcal{J}_{g,P}^{\underline{0}}\subset \jbgp(\phi)$ holds. Note that $U(\phi)$ always contains the moduli stack $\mathcal{M}_{g,P}^{\mathrm{ct}}$ of curves of compact type.

The right hand side of \eqref{overct} can be computed in terms of standard tautological classes by applying Theorem~\ref{MainThm} in combination with \eqref{inversion}. For $\ell=0$, the left hand side of \eqref{overct} has been computed in terms of standard tautological classes by Janda--Pandharipande--Pixton--Zvonkine \cite{MR3668650}. This produces lots of explicit relations in the tautological ring of the moduli stack $\mathcal{M}_{g,P}^{\mathrm{ct}}$ of curves of compact type.  We do not know if there is any reason to expect that these relations should be expressible as linear combinations of known ones, i.e. Pixton's relations proven in \cite{PPZ} and \cite{Janda}.

 Relation~\eqref{overct} is also valid over $\mbgp$ for some choices of $\ell, d_P$. In \cite[Proposition~14]{HKP} the authors observed that  $U (\phi)$ coincides with $\mbgp$  if and only if $\ell=0$ and $d_P=e_i - e_j$ for some $i,j \in P$, where $e_t\colon \, P \to \mathbb{Z}$ is defined by \[e_t(p)=\begin{cases} \displaystyle 1 &  \textrm{ when } t =p\\ 0 & \textrm{otherwise.} \end{cases} \]

For $\mathsf D_{i,j}= \sigma_i - \sigma_j$, Relation~\eqref{overct} becomes
\be\label{explicit}
\mathsf{DRC}(\ell=0, e_i - e_j)= c_{g} (-R^\bullet {\pi}_\ast \mathsf{D}_{i,j}(\phi)) \,\,\in\,\, R^{g}(\mbgp),
\ee
which again is valid whenever $\phi$ is such that the inclusion  $\mathcal{J}_{g,P}^{\underline{0}}\subset \jbgp(\phi)$ holds.  Explicitly, the modified divisor $\mathsf{D}_{i,j}(\phi)$ equals
\[
\mathsf D_{i,j}(\phi) = \sigma_i - \sigma_j - \sum_{(h,S): i \in S, j \notin S} C_{h,S} + \sum_{(h,S): j \in S, i \notin S} C_{h,S}.
\]
Again, the right hand side of \eqref{explicit} is computed by combining Theorem~\ref{MainThm} with \eqref{inversion}, and the left hand side was calculated in \cite{MR3668650}. Using \cite{admcycles} we have verified that the ensuing relation of standard tautological classes can be expressed as a linear combination of Pixton's relations for all $g \leq 4$. This also provides a non-trivial check of our formula in Theorem~\ref{MainThm}. Again, we do not know of an a priori reason to expect these relations to follow from Pixton's, except when $i=j$ where the right hand side of \eqref{explicit} simply equals $\lambda_g$.

%%%%%%%%%%%%%%%%%%%%%%%%%%%%%%%%%%
\section{Open problems}

We conclude the paper with a list of natural open questions.

\subsection{Is $\mathsf Z(\phi)$ tautological?}

Formula~\eqref{chernclasses} implies that the restriction of each class $\mathsf Z(\phi)$ to $U(\phi)$ is tautological on $U(\phi)$ -- meaning that it is the restriction to $U(\phi)$ of a tautological class globally defined on $\mbgp$. That tautological class is explicitly expressed in terms of decorated boundary strata by combining Theorem~\ref{MainThm} with Formulas~\eqref{chernclasses} and~\eqref{inversion}.  We do not know whether the class $\mathsf Z(\phi)$ is, in general, itself tautological on $\mbgp$, although we do expect that this should be the case.
Except for when $\mathsf Z(\phi)$ has codimension $1$ or $2$ (when we know that the entire cohomology of $\mbgp$ is tautological), the only classes  $\mathsf Z(\phi)$ that we know to be tautological on $\mbgp$ for general $g$ and $P$ are those for $r=d=k=0$ and $\phi$ a small perturbation of $\underline{0} \in V_{g,P}^0$. This follows from the main result of \cite{HKP}, showing that this class coincides with the \emph{Double Ramification Cycle}, see Section~\ref{drc2}. The latter is shown to be tautological in \cite{FPtaut}.

%%%%%%%%%%%%%%%%%%%%%%%%%%%%%%%%%%%%%%%
\subsection{Wall-crossing}

For fixed $d \in \mathbb{Z}$ and for every choice of nondegenerate elements $\phi$ and $\phi'$ of $V_{g,P}^d$ one has classes $\mathsf w^r_d(\phi) \in A^\bullet(\jbgp(\phi))$ and $\mathsf w^r_d(\phi') \in A^\bullet(\jbgp(\phi'))$. A natural question is to ``compute'' (in terms of some  natural classes) the difference 
\[
\mathsf w^r_d(\phi) - \alpha^*(\mathsf w^r_d(\phi')) \, \in  A^{g-\rho}(\jbgp(\phi)),
\] where $\alpha$ is any birational isomorphism $\jbgp(\phi) \dashrightarrow \jbgp(\phi')$ that commutes with the forgetful morphisms to $\mbgp$ (such birational maps are esplicitly characterised in \cite[Section~6.2]{KassPagani1}). To the best of our knowledge, this question has been answered only for the case of the theta divisor, namely when $r=0$ and $d=g-1$, in \cite[Theorem~4.1]{KassPagani0}. 

Another natural question is to compute the difference of the pullbacks 
\be\label{differenceZ}
\mathsf Z^r_{\ell, d_P}(\phi) - \mathsf Z^r_{\ell', d_P'}(\phi') \ \in A^{g-\rho}(\mbgp)
\ee
for different assignments $(\ell, d_P)$, $(\ell', d_P')$ such that $\ell(2g-2) \sum_{p \in P} d_p=\ell'(2g-2) \sum_{p \in P} d_p'=d$ and different nondegenerate $\phi$, $\phi' \in V_{g,P}^d$. The case of the pullback of the theta divisor is again covered explicitly in \cite[Theorem~5.1]{KassPagani0}.  Theorem~\ref{MainThm} immediately allows us to generalise the result in loc.~cit., in the sense that it computes explicitly, in terms of decorated boundary strata classes of $\mbgp$, the difference \eqref{differenceZ}, whenever $\phi$ and $\phi'$ are such that the corresponding Abel--Jacobi sections $\mathsf s$ and $\mathsf {s'}$ extend to morphisms on $\mbgp$. Example~\ref{esempio} checks that the results of this paper match the earlier results of \cite{KassPagani0} for the case of the pullback of the theta divisor.

\medskip

\subsection*{Acknowledgements}
NP would like to thank Jesse Kass for his crucial insight on extending Brill--Noether classes to compactified universal Jacobians. He is also very grateful to Nicola Tarasca for many useful discussions. AR wishes to thank Filippo Viviani for providing several useful insights on the topic of universal Brill--Noether theory, and Bashar Dudin for many fruitful conversations about compactified Jacobians.

NP acknowledges support from the EPSRC First Grant Scheme  EP/P004881/1  with title ``Wall-crossing on universal compactified Jacobians''. AR wishes to thank Max-Planck Institut f\"{u}r Mathematik for support. Jason van Zelm was supported by the Einstein Foundation Berlin during the course of this work.

\bigskip
%\clearpage
\bibliographystyle{amsplain}
\bibliography{bib}

\end{document}